\documentclass[12pt, reqno]{amsart}

\usepackage{amsthm, amssymb, amsmath}
\usepackage{fullpage}
\usepackage{verbatim}
\usepackage{graphicx}
\usepackage[all]{xy}
\usepackage{algorithm}
\usepackage{multirow}
\usepackage{rotating}
\usepackage{lscape}

\usepackage[usenames,dvipsnames]{color}

\newtheorem{thm}{Theorem}[section]

\newtheorem{lemma}[thm]{Lemma}
\newtheorem{cor}[thm]{Corollary}
\newtheorem*{cor*}{Corollary}
\newtheorem{prop}[thm]{Proposition}

\newtheorem*{conjecture*}{Conjecture}

\newtheorem*{RedLem*}{Reduction Lemma}

\theoremstyle{remark} 
\newtheorem*{question*}{Question}

\newtheorem{remark}[thm]{Remark}

\theoremstyle{definition} 

\newtheorem*{define*}{Definition}

\numberwithin{equation}{section}  

\newcommand{\OO}{\mathcal{O}}    
\newcommand{\FF}{\mathbb{F}}      
\newcommand{\F}{\mathbb{F}}      
\newcommand{\ZZ}{\mathbb{Z}}     
\newcommand{\Z}{\mathbb{Z}}     
\newcommand{\RR}{\mathbb{R}}     
\newcommand{\PP}{\mathbb{P}}      
\newcommand{\QQ}{\mathbb{Q}}      
\newcommand{\Q}{\mathbb{Q}}      
\newcommand{\CC}{\mathbb{C}}      

\newcommand{\frakp}{\mathfrak{p}}
\newcommand{\fraka}{\mathfrak{a}}

\newcommand{\Spec}{\operatorname{Spec}}
\newcommand{\ord}{\operatorname{ord}}
\newcommand{\Gal}{\operatorname{Gal}}  
\newcommand{\an}[1]{\operatorname{an}}  
\newcommand{\Aut}{\operatorname{Aut}}   
\newcommand{\Rat}{\operatorname{Rat}}    
\newcommand{\PGL}{\operatorname{PGL}}

\newcommand{\alg}[1]{{\overline{#1}}}
\newcommand{\GG}{\mathbb{G}}

\newcommand{\tth}{^{\operatorname{th}}}
\newcommand{\Berk}{\mathbf{P}}  

\newcommand{\Manoa}{M\=anoa}
\newcommand{\Hawaii}{Hawai\kern.05em`\kern.05em\relax i}

\newcommand{\red}{\mathrm{red}}  
\newcommand{\Res}{\mathrm{Res}}  

\newcommand{\Alg}{\text{-}\mathbf{Alg}}
\newcommand{\Grp}{\mathbf{Grp}}

\newcommand{\mat}[1]{\left( \begin{smallmatrix} #1 \end{smallmatrix} \right)}
\newcommand{\oo}{\mathfrak{o}}
\newcommand{\tab}{\hspace{.4cm}}
\newcommand{\Fix}{\operatorname{Fix}}

\newcommand{\FFF}{\overline{\mathbb{F}}}

\newcommand{\per}{\mathrm{per}}
\newcommand{\Conj}{\mathrm{Conj}}

\title[Computing conjugating sets]{Computing conjugating sets and automorphism groups of rational functions}

\author{Xander Faber}
\author{Michelle Manes}

\address{Department of Mathematics,
University of \Hawaii\ at \Manoa, 
Honolulu, HI}
\email{xander@math.hawaii.edu, mmanes@math.hawaii.edu}
\urladdr{http://www.math.hawaii.edu/\~{}xander/, http://www.math.hawaii.edu/\~{}mmanes/}

\author{Bianca Viray}
\address{Department of Mathematics,
Brown University,
Providence, RI}
\email{bviray@math.brown.edu}
\urladdr{http://math.brown.edu/\~{}bviray}
\date{\today}
 
\thanks{The first and second authors were partially supported by NSF grants DMS-0902532 and DMS-1102858, respectively.  The third author was partially supported by NSF grant DMS-1002933 and by ICERM}


\begin{document}

	\begin{abstract}
		%
		%

%
Let $\phi$ and $\psi$ be endomorphisms of the projective line of degree at least~$2$, defined over a noetherian commutative ring $R$ with unity.  From a dynamical perspective, a significant question is to determine whether $\phi$ and $\psi$ are conjugate (or to answer the related question of whether a given map $\phi$ has a nontrivial automorphism).  We show that the space of automorphisms of~$\PP^1$ conjugating $\phi$ to $\psi$ is a finite subscheme of $\PGL_2$ (respectively that the
automorphism group of~$\phi$ is a finite group scheme).  

We construct efficient algorithms for computing the set of conjugating maps (resp. the group of automorphisms) when $R$ is a field.  Each of our algorithms takes advantage of different dynamical structures, so context (e.g., field of definition and degree of the map) determines the preferred algorithm.
We have implemented them in Sage when $R$ is a finite field or the field of rational numbers, and 
we give running times for computing automorphism groups for hundreds of random endomorphisms of $\PP^1$. 
These examples  demonstrate the superiority of these new algorithms over a na\"ive approach using Gr\"obner bases.  
\end{abstract}

	\maketitle

	\section{Introduction}

	Let $F$ be a field, and let $\phi \in F(z)$ be a rational function.  Let $f,g\in F[z]$ be relatively prime polynomials with $\phi = f/g$.  Unless otherwise specified, we assume throughout that $d = \deg(\phi) := \max\{\deg(f), \deg(g)\}$ is at least~$2$.  When viewed as an endomorphism of the projective line $\PP^1_F \stackrel{\phi}{\to} \PP^1_F$, a dynamical theory of $\phi$ arises from iteration. That is, for $x \in \PP^1(F)$, we may consider its orbit
		\[
			x\mapsto \phi(x)\mapsto \phi^2(x)\mapsto \phi^3(x)\mapsto\cdots
		\]
	(Here we write $\phi^1 = \phi$ and $\phi^n = \phi \circ \phi^{n-1}$ for each $n > 1$.) The case $F = \CC$ --- dynamics of self-maps of the Riemann sphere --- has a fascinating history dating back as far as Newton \cite{Alexander_Dynamics_History, Milnor_Dynamics_Book_2000}. When $F$ is a finite field, these dynamical systems behave (conjecturally) like random maps, which has applications to factoring integers \cite{Pollard_rho_1975, Bach_Pollard_1991}. When $F$ is a number field, then we have the younger theory of arithmetic dynamics \cite{Silverman_Dynamics_Book_2007}. The case where $F$ is a non-Archimedean field is younger still and draws much inspiration from the complex case \cite{Baker-Rumely_BerkBook_2010, Jonsson_Berkovich_Notes}. 

	In the present paper, we study the following pair of algorithmic problems:
	\begin{enumerate}
		\item\label{Problem 1} for two rational functions $\phi$ and $\psi$, determine the set of rational functions $s$ of degree~$1$ (automorphisms of $\PP^1$)
			that conjugate $\phi$ to $\psi$,  i.e., such that
			 $s\circ \phi = \psi\circ s$.
		\item\label{Problem 2} for a given rational function $\phi$, determine the \textbf{automorphism group of $\phi$}; 
			i.e., determine the set of rational functions $s$ of degree $1$ such that $s\circ \phi \circ s^{-1} = \phi$.
	 \end{enumerate}

	If two rational functions $\phi, \psi \in F(z)$ are conjugate, then they exhibit the same \textit{geometric} dynamical behavior.  Indeed, if $\alg{F}$ is an algebraic closure of $F$ and $s \in \alg{F}(z)$ conjugates $\phi$ to $\psi$, then $s$ maps the $\phi$-orbit of a point $x \in \PP^1(\alg{F})$ to the $\psi$-orbit of $s(x)$.  Conversely, given two functions that seemingly exhibit the same geometric dynamical behavior, one wants to know if they are conjugate, or if there is some deeper structure that should be investigated.

	We say that $\phi$ and $\psi$ are conjugate over a field extension $E/F$ if they satisfy the relation $s \circ \phi = \psi \circ s$ for some rational function $s \in E(z)$ of degree~1. In this case, they have the same \emph{arithmetic} dynamical behavior over $E$; e.g., $s$ maps $\phi$-orbits of $E$-rational points to $\psi$-orbits of $E$-rational points, and the field extension of $E$ generated by the period-$n$ points of $\phi$ and $\psi$ must agree for every $n\geq 1$. The automorphism group of $\phi$ (or $\psi$) bounds the size of the field extension generated by the coefficients of any conjugating map \cite{LMT_Bounds_for_Twists_2012}. As a special case,  Galois cohomology shows that  conjugacy over $F$ and conjugacy over an algebraic closure $\alg{F}$ are equivalent notions whenever $\phi$ and $\psi$ have trivial automorphism group (over $\alg{F}$). 		

	The symmetry locus of rational functions (i.e., the space of rational functions with non-trivial automorphism group) can be thought of as an analogue of the locus of abelian varieties that have extra automorphisms.  Indeed, just as the presence of elliptic curves with extra automorphisms obstructs the existence of a universal elliptic curve, so does the symmetry locus obstruct the existence of a fine moduli space of conjugacy classes of rational functions.

	It is worth noting that modern research in dynamics often focuses on low degree, with an abundance of open questions even in degrees~$2$ (see, for example, \cite{Milnor_Rational_Maps_1993, Poonen_Preperiodic_Classification_1998}) and~$3$ (see~\cite{Milnor_Cubics_2009}), and virtually nothing at all known in degrees greater than $3$.  So even for degrees  smaller than $10$, there is ample room for generating, testing, and refining new conjectures, and we believe the tools presented here will be useful in this regard.  

	In Sections~\ref{Sec: Red Lemma}--\ref{Sec: Conjugate}, we define two functors $\underline{\Conj}_{\phi, \psi}$ and $\underline{\Aut}_\phi$ and show that they are represented by schemes $\Conj_{\phi,\psi}$ and $\Aut_{\phi}$; the points of $\Conj_{\phi,\psi}$ and $\Aut_{\phi}$ are precisely the set of rational functions of degree~1 that conjugate $\phi$ to $\psi$ and the automorphism group of $\phi$, respectively. To provide solutions to problems \eqref{Problem 1} and \eqref{Problem 2} above in the case of number fields, we will work with conjugating sets and automorphism groups over global fields and their reductions at several primes.

In the remainder of the introduction, we describe some of our results for the scheme $\Aut_\phi$ and indicate how they will give rise to algorithms for computing its field-valued points; the analogous results for $\Conj_{\phi,\psi}$ are stated in~\S\ref{Sec: Conjugate}. For notation, let $R$ be a noetherian commutative ring with unity, and let $R\Alg$ and $\Grp$ denote the categories of commutative $R$-algebras and (arbitrary) groups, respectively. For any $R$-algebra $S$, we identify $\PGL_2(S)$ with $\Aut(\PP^1_S)$, the group of automorphisms of $\PP^1$ defined over $S$. We make the following definition:

	\begin{define*}
		Let $\phi: \PP^1_R \to \PP^1_R$ be a nonconstant morphism. Let $\underline{\Aut}_\phi$ denote the functor from $R\Alg \to \Grp$ that sends 
		\[
			\underline{\Aut}_\phi\colon S \mapsto \{f \in \Aut(\PP^1_S) : 
			\phi = \phi^f := f \circ \phi \circ f^{-1} \}. 
		\]
	\end{define*}

	If $k$ is a non-Archimedean field (not necessarily complete) with valuation ring $\oo$, we say that an endomorphism $\phi: \PP^1_k \to \PP^1_k$ has \textbf{good reduction} if there exists a morphism $\Phi: \PP^1_\oo \to \PP^1_\oo$ that agrees with $\phi$ on the generic fiber.  In \S\ref{Sec: Red Lemma} we prove the following.

		\begin{RedLem*}
			Let $k$ be a non-Archimedean field with valuation ring $\oo$ and residue field $\FF$, 
			and let $\phi \in k(z)$ be a rational function of degree at least~$2$ \textup(which is equivalent to 
			a morphism $\PP^1_k \to \PP^1_k$\textup).
			Suppose that $\phi$ has good reduction. Then every element of 
			$\underline{\Aut}_\phi(k)$ has good reduction, and the canonical reduction $\oo \to \FF$ induces 
			 a homomorphism $\red : \underline{\Aut}_\phi(k) \to \underline{\Aut}_\phi(\FF)$. If $\FF$ has characteristic $p > 0$ 
			 \textup(resp. characteristic zero\textup), then the kernel of reduction is a $p$-group \textup(resp. 
			 trivial\textup). 
		\end{RedLem*}

	We use this result to deduce that $\underline{\Aut}_\phi$ is representable and, when $\phi$ has degree at least~2, that the representing scheme is proper.  More generally, in \S\ref{Sec: Aut Properties} we prove:
	\begin{thm}	\label{Thm: Finite group scheme}
		Let $R$ be a commutative ring and let $\phi: \PP^1_R \to \PP^1_R$ be a nonconstant endomorphism. Then the functor $\underline{\Aut}_\phi$ is represented by a closed $R$-subgroup scheme $\Aut_\phi \subset \PGL_2$. If moreover $\deg(\phi) \geq 2$, then $\Aut_\phi$ is finite over $\Spec R$. 
	\end{thm}

	\begin{remark}
		The group scheme $\Aut_\phi$ need not be flat over $\Spec R$. For example, if $\phi(z) = z^2$ as an endomorphism of the projective line over $\ZZ_2$, then there is nontrivial $2$-torsion in the ring of global functions of $\Aut_\phi$. Intuitively, this is because $\Aut_\phi(\alg{\QQ_2}) = \{z, 1/z\} $, while $\Aut_\phi(\FFF_2) \cong \PGL_2(\FF_2)$, which has order~6. 
	\end{remark}

	\begin{remark}
		If $\deg(\phi) = 1$, then $\Aut_\phi$ is not a finite group scheme in general. Consider the examples $\phi(z) = z$ and $\psi(z) = z + 1$, for which $\Aut_\phi = \PGL_2$ and $\Aut_\psi \cong \GG_a$, respectively.
	\end{remark}

	If $K$ is a number field and $v$ is a finite place of $K$, we write $K_v$ and $\FF_v$ for the completion of $K$ at $v$ and the residue field of $K_v$, respectively. If $\phi \in K(z)$ is a rational function, we say that it has \textbf{good reduction at $v$} if the induced rational function over $K_v$ has good reduction in the above sense. 
(Equivalently, $\phi$ has good reduction at $v$ if one can reduce its coefficients modulo $v$, and the resulting endmorphism of $\PP^1_{\FF_v}$ has the same degree as $\phi$.)  Another application of the Reduction Lemma proved in Section~\ref{Sec: Red Lemma} gives us an injectivity statement away from finitely many places of $K$:

	\begin{prop}\label{prop:injectivity}
		Let $K$ be a number field and let $\phi \in K(z)$ be a rational function of degree $d \geq 2$. Define $S_0$ to be the set of rational primes given by 
		 \[
		 	S_0 = \{2\} \cup 
				\left\{p \text{ odd}: \frac{p - 1}{2} \Big| [K : \QQ] \ \text{ and } \ p \mid d(d^2 - 1) \right\},
		\]
		and let $S$ be the \textup(finite\textup) set of places of $K$ of bad reduction for $\phi$ along with the places that divide a prime in $S_0$. Then $\red_v: \Aut_{\phi}(K) \to \Aut_{\phi}(\FF_v)$ is a well defined injective homomorphism for all places $v$ outside $S$. 
	\end{prop}

	Proposition~\ref{prop:injectivity} often allows one to determine the group structure of $\Aut_\phi(K)$ very quickly by computing $\Aut_\phi(\FF_v)$ for a few places $v \not\in S$.  This is analogous to the way one typically computes the torsion subgroup of an elliptic curve over a number field; see \cite[VII.3]{Silverman_AEC_2009}.  If one wishes to compute the \emph{elements} of $\Aut_{\phi}(K)$ rather than just the group structure, then more work is required. 

	In Section~\ref{Sec: Conjugate}, we define the scheme $\Conj_{\phi, \psi}$ and prove modified versions of the Reduction Lemma, Theorem~\ref{Thm: Finite group scheme}, and Proposition~\ref{prop:injectivity} in this setting.

	In \S\ref{Sec: GeneralAlgorithms}, we give algorithms that can be used to compute the field-valued points of both $\Aut_{\phi}$ and $\Conj_{\phi,\psi}$. By way of a disclaimer, we have attempted to stress concept and clarity in our algorithms; we have not endeavored to explain all of the small tricks we used at the level of implementation.  This is especially true in Algorithm~\ref{Alg:NumberField}.  We refer the interested reader to our source code which is included with the \texttt{arXiv} distribution of this article.

	The first new algorithm that we develop, the method of invariant sets, computes the absolute automorphism group or absolute conjugating set.  We explain how a slight modification gives a method to compute the automorphism group or conjugating set defined over a fixed finite ground field.  In~\S\ref{Sec: Number Fields}, we prove that the height of the elements in $\Conj_{\phi, \psi}$ are bounded in terms of the coefficients of $\phi$ and $\psi$.  This allows us to develop a Chinese Remainder Theorem (CRT) algorithm to compute the $K$-rational points of $\Conj_{\phi,\psi}$ and $\Aut_{\phi}$ when $K$ is a number field. In principal, the CRT algorithm is exactly the same regardless of whether one is computing $\Aut_{\phi}(K)$ or $\Conj_{\phi, \psi}(K)$; however, in practice, each admits several distinct ways to increase efficiency.  See the end of \S\ref{Sec: Number Fields} for some examples in the case of computing automorphism groups and \S\ref{sec: early termination} for the case of ~$\Conj_{\phi, \psi}$.

	One important difference between the sets $\Aut_{\phi}(F)$ and $\Conj_{\phi, \psi}(F)$ is that the action of $\phi$ on the fixed points of $s\in\Aut_{\phi}(F)$ is highly restricted, both geometrically and arithmetically.  To the best of our knowledge, no analogous fact holds for the action of $\phi$ or $\psi$ on the fixed points of $s\in\Conj_{\phi, \psi}(F)$.  We exploit this restriction on the fixed points of $s\in\Aut_{\phi}(F)$ to develop another algorithm for computing $\Aut_{\phi}(F)$: the method of fixed points; this is described in \S\ref{Sec: AutAlgorithms}

	In~\S\ref{Sec: Examples} we compare the running times of the algorithms for computing $\Aut_{\phi}(\QQ)$ on rational functions with non-trivial automorphism group and on a large number of randomly generated rational functions of various degrees and various heights.  The running times demonstrate that the method of fixed points is preferable to the CRT method for rational functions of degree up to about 12, when the two methods become comparable.  For larger degrees the CRT method is preferable.  For the sake of completeness, we also compare our algorithms to the ``na\"ive'' algorithm using Gr\"obner bases.  Our experiments show that the Gr\"obner basis method is comparable with the fixed point method when the degree is two or three, but quickly becomes impractical.  

	\section*{Acknowledgements}

		This project began at the University of Georgia, during an NSF--sponsored summer school on Arithmetic Dynamics.  We thank the organizer, Robert Rumely, for the experience. The authors are grateful for the opportunity to complete the project at the Institute for Computational and Experimental Research in Mathematics. We thank the anonymous referees of
an earlier draft for their  comments,
 and give special thanks to Joseph H. Silverman for helpful comments on our number field algorithm. 

	\section{Proof of the Reduction Lemma}\label{Sec: Red Lemma}

		For background on the Berkovich projective line and dynamics, see~\cite{Baker-Rumely_BerkBook_2010}.  For a more concise summary of the ideas used  in this section, we direct the reader to \cite{Faber_Berk_RamI_2011}. In this section we write $\Aut_\phi(F)$ instead of $\underline{\Aut}_\phi(F)$, an abuse of notation which will be properly justified when we prove that the scheme $\Aut_\phi$ exists in the next section.

		Let $\CC_k$ be the completion of an algebraic closure of the completion of $k$, and let $\Berk^1$ be the Berkovich analytification of the projective line $\PP^1_{\CC_k}$. The morphism $\phi$ extends functorially to~$\Berk^1$. We use two key facts due to Rivera-Letelier \cite[Thm.~4]{Rivera-Letelier_Periodic_Points_2005}: 
		\begin{enumerate}
			\item a rational function $f$ has good reduction if and only if the Gauss point $\zeta \in \Berk^1$ is totally invariant; i.e., $f^{-1}(\zeta) = \{\zeta\}$, and 
			\item a rational function of degree at least~2 has at most one totally invariant point in $\Berk^1 \smallsetminus \PP^1(\CC_k)$.  
	 	\end{enumerate}

		For $f \in \Aut_\phi(k)$, we have
		\begin{align*}
			f^{-1}(\zeta) = f^{-1} (\phi^{-1}(\zeta))  
			= (\phi \circ f)^{-1}(\zeta) = 
			(f \circ \phi)^{-1}(\zeta) = \phi^{-1}(f^{-1}(\zeta)). 
		\end{align*}
		Hence $f^{-1}(\zeta)$ is a totally invariant type~II point for $\phi$, so that $f(\zeta) = \zeta$. Equivalently, $f$ has good reduction. Thus the reduction map $\red: \Aut_\phi(k) \to \Aut_\phi(\FF)$ is well-defined, and it is evidently a homomorphism.

		Now we compute the kernel of reduction. Suppose $\red(f)$ is trivial. Without loss of generality, we may replace $k$ with a finite extension in order to assume that $f$ has a $k$-rational fixed point. Moreover, we may 
	conjugate $f$ by an element of $\PGL_2(\oo)$ in order to assume that $f(\infty) = \infty$. Now $f(z) = \alpha z + \beta$. If $f$ has order $m > 1$, then the equation $f^m(z) = f(z)$ shows that $\alpha$ is an $m$-th root of unity. But $\red(f)$ is trivial, so we have $\tilde \alpha = 1$. If $k$ has residue characteristic zero, then we conclude that $\alpha = 1$ and $\beta = 0$. Otherwise, we find that $\alpha$ is a $p$-power root of unity in $k$, and hence $f$ has $p$-power order in $\Aut_\phi(k)$. The proof of the Reduction Lemma is complete.

	\begin{remark}
		A different proof of the first part of the Reduction Lemma can be given using the maximum modulus principle in non-Archimedean analysis \cite[Lem.~6]{Petsche_Szpiro_Tepper_2009}. 
	\end{remark}

	\begin{prop}\label{Prop: Geometric Bound}
		Let $F$ be a field, and let $n \geq 2$ be an integer. Suppose that $\phi: \PP^1_F \to \PP^1_F$ is a morphism of degree $d \geq 2$ such that $\Aut_\phi(F)$ contains an element of order $n$. Then $n$ divides one of $d$, $d - 1$, or $d + 1$. 
	\end{prop}

	\begin{proof}
		We may assume without loss of generality that $F$ is algebraically closed. Let $s \in \Aut_\phi(F)$ have order~$n$. We conjugate one of the fixed points of $s$ to $\infty$, so that $s = \mat{ \alpha & \beta \\ & 1}$. (Note that replacing $s$ with $usu^{-1}$ has the effect of replacing $\phi$ with $u  \phi  u^{-1}$.) The proof divides into two cases, depending on whether $s$ has one or two fixed points. 

		If $s$ has only one fixed point, then necessarily $n = \mathrm{char}(F)$ is prime and $\alpha = 1$. (See, for example, \cite[Lem.~3.1]{PGL(2)_subgroups}.) Replace $s$ with $\mat{ \beta^{-1} & \\ & 1} s \mat{ \beta & \\ & 1}$ in order to assume that $\beta = 1$. It follows that $\phi(z+1) - 1 = \phi(z)$, or equivalently, that the function $\phi(z) - z$ is invariant under the map $z \mapsto z + 1$. Hence there exists a rational function $\psi(z) \in F(z)$ such that $\phi(z) - z = \psi(z^n - z)$. We conclude that $\deg(\phi) = n \cdot \deg(\psi)$ or $n \cdot \deg(\psi) + 1$. 

		Now suppose that $s$ has two distinct fixed points: $\infty$ and $\beta / (1 - \alpha)$. We may conjugate the second fixed point to $0$ in order to assume that $\beta = 0$. Note that this implies that $\alpha \in F^\times$ has multiplicative order~$n$. To say that $s$ is an automorphism of $\phi$ is equivalent to saying that $\phi(z) / z$ is invariant under the map $z \mapsto \alpha z$. Hence there is a rational function $\psi \in F(z)$ such that $\phi(z)/z =\psi(z^n)$. So $\deg(\phi) = n \cdot \deg (\psi)$ or $n \cdot \deg(\psi) \pm 1$.
	\end{proof}

	\begin{proof}[Proof of Proposition~$\ref{prop:injectivity}$]
		By the Reduction Lemma, it suffices to prove that if $v \not\in S$, then $\Aut_\phi(K)$ has no element of order $p$, where $v \mid p$. Suppose otherwise.

		The group $\PGL_2(K)$ contains an element of order $p$ if and only if $\zeta_p + \zeta_p^{-1} \in K$ for some primitive $p$-th root of unity $\zeta_p$ \cite{Beauville_Finite_Subgroups_2010}.  Note that $[\Q(\zeta_p + \zeta_p^{-1}):\Q] = \frac12(p - 1)$ for $p>2$, so that $\frac{p-1}{2} \mid [K:\QQ]$. 
	If $\Aut_\phi(K)$ contains an element of order $p$, then $p$ divides $d(d^2 - 1)$ by Proposition~\ref{Prop: Geometric Bound}. Hence $p \in S_0$, and so $v \in S$. 
	\end{proof}

\section{Proof of Theorem~\ref{Thm: Finite group scheme}}
\label{Sec: Aut Properties}

	Fix a commutative ring $R$. Over $R$,  $\PGL_2$ may be embedded as an affine subvariety of $\PP^3_R = \mathrm{Proj} \ R[\alpha, \beta, \gamma, \delta]$; indeed, it is the complement of the quadric $\alpha\delta - \beta\gamma = 0$. Let $\phi: \PP^1_R \to \PP^1_R$ be a nonconstant endomorphism. We may define $\Aut_\phi$ as a subgroup scheme of $\PGL_2$ as follows. After fixing coordinates of $\PP^1_R$, the morphism $\phi$ can be given by a pair of homogeneous polynomials $\Phi = (\Phi_0(X,Y), \Phi_1(X,Y))$ of degree $d = \deg(\phi)$ with coefficients in $R$ such that the homogeneous resultant $\Res(\Phi_0, \Phi_1)$ is a unit in $R$. The pair $\Phi_0, \Phi_1$ is unique up to multiplication by a common unit in $R$. Similarly, for any $R$-algebra $S$, an element $f \in \PGL_2(S)$ may be given by a pair $F = (\alpha X + \beta Y, \gamma X + \delta Y)$ with $\alpha, \beta, \gamma, \delta \in S$ and $\alpha\delta - \beta\gamma \in S^\times$. Note that $f^{-1}$ is represented by the pair $F^{-1} := (\delta X - \beta Y, -\gamma X + \alpha Y)$. Then $f \circ \phi \circ f^{-1} = \phi$ is equivalent to saying that $F \circ \Phi \circ F^{-1}$ and $\Phi$ define the same morphism on $\PP^1_S \to \PP^1_S$. If we define $(\Phi_0'(X,Y), \Phi_1'(X,Y)) = F \circ \Phi \circ F^{-1}$, then this means 
	\begin{equation}
	\label{Eq: Aut Defining Eqs}
		\Phi_0(X,Y) \Phi'_1(X,Y) - \Phi_1(X,Y) \Phi'_0(X,Y) = 0.
	\end{equation}
The expression on the left is a homogeneous polynomial of degree $2d$ in $X$ and $Y$ whose coefficients are homogeneous polynomials in $R[\alpha, \beta, \gamma, \delta]$. So \eqref{Eq: Aut Defining Eqs} gives $2d + 1$ equations that cut out a closed subscheme of $\PGL_2$ defined over $R$. One checks readily that $\Aut_\phi(S)$ is a subgroup of $\PGL_2(S)$ for every $S$.

	Next we argue that $\Aut_\phi$ is a finite group scheme over $R$ when $\phi$ has degree at least~2. The map $\Aut_\phi \to \Spec R$ is quasi-finite. Indeed, it suffices to check this statement on geometric fibers, and Silverman has shown that $\Aut_\phi(L)$ is a finite group for any algebraically closed field $L$ \cite[Prop.~4.65] {Silverman_Dynamics_Book_2007}.\footnote{Alternatively, \S\ref{Sec: Invariant Sets} gives a conceptually simpler proof of Silverman's result.} 

	Moreover, $\Aut_\phi$ is proper over $\Spec R$. Indeed, since $\Aut_\phi$ and $\Spec R$ are noetherian, this can be checked using the valuative criterion for properness using only discrete valuation rings \cite[Ex.~II.4.11]{Hartshorne_Bible}. Let $\oo$ be a discrete valuation ring with field of fractions $k$, and consider a commutative diagram
	 	\[
	 		\xymatrix{
	 			\Spec k \ar[r] \ar[d] & \Aut_\phi \ar[d] \\
				\Spec \oo \ar[r] \ar@{-->}[ur] & \Spec R.
	 		}
	 	\]
	 The left vertical map is the canonical open immersion, and the right vertical map is the structure morphism. We must show there is a unique morphism $\Spec \oo \to \Aut_\phi$ that makes the entire diagram commute. Without loss of generality, we may assume that $R = \oo$ and that the lower horizontal arrow is the identity map. 

 	If $v: k  \to \ZZ \cup \{ +\infty \}$ is the canonical extension of the valuation on $\oo$, then we may endow $k$ with the structure of a non-Archimedean field by setting $|x| = e^{-v(x)}$ for every $x \in k$. (Note that we interpret $e^{-\infty}$ as zero.) Since $\phi$ is defined over $\oo$, it has good reduction. The Reduction Lemma asserts that every $k$-automorphism of $\phi$ also has good reduction. Equivalently, every $k$-valued point may be extended to an $\oo$-valued point, which is what we wanted to show. 

	We now know that $\Aut_\phi \to \Spec R$ is a quasi-finite proper morphism. Zariski's main theorem tells us that it factors as an open immersion of $R$-schemes $\Aut_\phi \to X$ followed by a finite morphism $X \to \Spec R$. But $\Aut_\phi$ is proper, so any open immersion is actually an isomorphism. Hence $\Aut_\phi$ is finite over $\Spec R$. This completes the proof of Theorem~\ref{Thm: Finite group scheme}.

\section{The conjugation scheme}
\label{Sec: Conjugate}

		Let $F$ be a field, and let $\phi, \psi: \PP^1_F \to \PP^1_F$ be a pair of endomorphisms of the projective line. In this section we want to describe the set of solutions $f \in \Aut(\PP^1)$ to the functional equation $f \circ \phi \circ f^{-1} = \psi$. That is, we investigate the question, ``What can be said about the space of all functions $f$ that conjugate $\phi$ to $\psi$?'' Setting $\psi = \phi$ recovers the automorphism group of $\phi$ as a special case of this question, and yet we need to do little extra work to answer the more general query.

	\begin{define*}
		Fix a non-negative integer $d$, and let $\phi, \psi: \PP^1_R \to \PP^1_R$ be two endomorphisms of degree $d$.  Write $\mathbf{Set}$ for the category of sets. The \textbf{conjugation scheme} of the pair $(\phi, \psi)$ is the $R$-scheme $\Conj_{\phi, \psi}$ representing the functor 
		$\underline{\Conj}_{\phi, \psi} : R\Alg \to \mathbf{Set}$ defined
		by 
			\[
				\underline{\Conj}_{\phi, \psi}(S) = \{f \in \Aut(\PP^1_S) :  f \circ \phi \circ f^{-1} = \psi \}. 
			\]
	\end{define*}

	\begin{thm} \label{Thm: Conj Scheme}
		Let $R$ be a commutative ring, let $d \geq 0$ be an integer, and let $\phi, \psi: \PP^1_R \to \PP^1_R$ be endomorphisms of degree~$d$. Then the functor $\underline{\Conj}_{\phi, \psi}$ is represented by a closed $R$-subscheme $\Conj_{\phi, \psi} \subset \PGL_2$. If moreover $d \geq 2$, then $\Conj_{\phi, \psi}$ is finite over $\Spec R$.  
	\end{thm}

\begin{remark}
	The theorem does not preclude the possibility that $\Conj_{\phi, \psi}$ is the empty scheme, which is typically the case when $d \geq 2$. The group scheme $\PGL_2$ has relative dimension~$3$ over $R$, while the space $\mathrm{Rat}_d$ of endomorphisms of $\PP^1$ of degree $d$ has relative dimension $2d+1 > 3$. So for a fixed $\phi \in \Rat_d(R)$, a general choice of $\psi$ will yield $\Conj_{\phi, \psi} = \varnothing$. 
\end{remark}

	\begin{remark} 
	When $\Conj_{\phi, \psi}$ is not the empty scheme, it is a principal homogeneous space for $\Aut_\phi$ (and $\Aut_\psi$). 
	\end{remark}
	
	The construction of the closed subscheme $\Conj_{\phi, \psi} \subset \PGL_2$ proceeds along the same lines as that of $\Aut_\phi$ in \S\ref{Sec: Aut Properties}, so we leave the details to the reader. Note that the equations defined by \eqref{Eq: Aut Defining Eqs} remain equally valid in this setting if we write $\Psi = (\Psi_0(X,Y), \Psi_1(X,Y))$ for a homogenization of $\psi$ and replace $\Phi_i$ with $\Psi_i$. In particular, $\Conj_{\phi, \psi}$ is cut out as a subscheme of  $\PGL_2$ by $2d + 1$ homogeneous polynomials of degree $d+1$ in the four variables $\alpha, \beta, \gamma, \delta$, where $\PGL_2 \subset \mathrm{Proj} \ R[\alpha, \beta, \gamma, \delta]$. 
	
	In the case $d \geq 2$ of the theorem, in order to establish that $\Conj_{\phi, \psi}$ is finite over $\Spec R$, one must argue that it is proper and quasi-finite. Properness follows from a direct generalization of the Reduction Lemma (and its proof):
	
		\begin{RedLem*}[Part II]
			Let $k$ be a non-Archimedean field with valuation ring $\oo$ and residue field $\FF$, 
			and let $\phi, \psi \in k(z)$ be rational functions of degree at least~$2$.
			Suppose that both $\phi$ and $\psi$ have good reduction. Then every element of 
			$\Conj_{\phi,\psi}(k)$ has good reduction, and the canonical reduction $\oo \to \FF$ induces 
			 a map of sets $\red_{\phi, \psi} : \Conj_{\phi,\psi}(k) \to \Conj_{\phi,\psi}(\FF)$. If the order of 
			 $\Aut_\phi(k)$ is relatively prime to the characteristic of $\FF$, then 
			 $\red_{\phi, \psi}$ is injective. 
		\end{RedLem*}
		
\begin{proof} 
	Only the final statement of the lemma requires further comment. The Reduction Lemma for $\Aut_\phi$ shows that the kernel of the  homomorphism $\red_\phi : \Aut_\phi(k) \to \Aut_\phi(\FF)$ is trivial. If $f, g \in \Conj_{\phi, \psi}(k)$ have the same image in $\Conj_{\phi, \psi}(\FF)$, then $f^{-1}g$ lies in the kernel of $\red_\phi$, so that $f = g$. 
\end{proof}

	To complete the proof of Theorem~\ref{Thm: Conj Scheme}, it remains to show that $\Conj_{\phi, \psi}$ is quasi-finite when $\deg(\phi) = \deg(\psi) \geq 2$, for which it suffices to take $R = F$ to be a field and prove that $\Conj_{\phi, \psi}(F)$ is finite. If $\Conj_{\phi, \psi}(F)$ is empty, we are finished. Otherwise, fix an element $f_0 \in \PGL_2(F)$ that conjugates $\phi$ to $\psi$. Given an element $f \in \Conj_{\phi, \psi}(F)$, we see that 
	\[
		(f_0^{-1} \circ f) \circ \phi \circ (f_0^{-1} \circ f)^{-1} = f_0^{-1} \circ \left( f \circ \phi \circ f^{-1} \right) \circ f_0 
			= f_0^{-1} \circ \psi \circ f_0 = \phi. 
	\]
That is, the association $f \mapsto f_0^{-1} \circ f$ defines a map of sets			
	\[
		\Conj_{\phi, \psi}(F) \to \Aut_\phi(F).
	\]
This map is evidently bijective. Since $\Aut_\phi(F)$ is a finite set \cite[Prop.~4.65] {Silverman_Dynamics_Book_2007}, so is $\Conj_{\phi, \psi}(F)$.   

	We close this section with a version of Proposition~\ref{prop:injectivity} that applies to conjugation sets. 
	
\begin{cor}\label{Cor: Injectivity}
	Let $K$ be a number field and let $\phi, \psi \in K(z)$ be rational functions of degree $d \geq 2$. Define $S_0$ to be the set of rational primes given by 
		 \[
		 	S_0 = \{2\} \cup 
				\left\{p \text{ odd}: \frac{p - 1}{2} \Big| [K : \QQ] \ \text{ and } \ p \mid d(d^2 - 1) \right\},
		\]
	and let $S$ be the \textup(finite\textup) set of places of $K$ of bad reduction for $\phi$ or $\psi$ along with the places that divide a prime in $S_0$. If $\Conj_{\phi, \psi}(K)$ is nonempty, then $\red_v: \Conj_{\phi, \psi}(K) \to \Conj_{\phi, \psi}(\FF_v)$ is a well defined injection of sets for all places $v$ outside $S$. 
\end{cor}

\begin{proof}
	Let $f_0 \in \Conj_{\phi, \psi}(K)$. For $v \not\in S$, we have the following diagram of morphisms of sets:
	\[
		\xymatrix{
			\Conj_{\phi,\psi}(K) \ar[r]^{\red_v} \ar[d]_{f_0^{-1} \circ} 
			& \Conj_{\phi, \psi}(\FF_v) \ar[d]^{\red_v(f_0)^{-1} \circ} \\
			\Aut_\phi(K) \ar[r]^{\red_v} & \Aut_\phi(\FF_v) \\
		}
	\]
The vertical arrows denote postcomposition with the indicated element; they are bijections by the discussion immediately preceding this proof. The Reduction Lemmas show that the horizontal arrows are well defined. The diagram commutes because $\PGL_2$ is a group scheme. We have already shown the lower horizontal arrow is injective (Proposition~\ref{prop:injectivity}), so the top one must share this property.
\end{proof}


\section{General Algorithms}
\label{Sec: GeneralAlgorithms}

	The goal of this section is to collect and compare a number of algorithms for computing the set $\Conj_{\phi, \psi}(F)$ over a variety of fields $F$, with an emphasis  on the cases where $F$ is a finite field or number field. Since $\Aut_\phi(F) = \Conj_{\phi, \phi}(F)$, these algorithms will also compute the automorphism group of a rational function $\phi$; in the next section we will propose routines for computing $\Aut_\phi(F)$ that are typically much more efficient. 
	
	Given two different rational functions $\phi$ and $\psi$, one might be interested in determining if they are conjugate over the field $F$ (or over an algebraic closure). One may use the methods below and include an early termination condition if any single conjugating map $f_0$ is found such that $\psi = f_0 \circ \phi \circ f_0^{-1}$. The proof that $\Conj_{\phi, \psi}(F)$ is finite in the preceding section shows that $\Conj_{\phi, \psi}(F) = f_0 \circ \Aut_\phi(F)$. If one wants to compute the full set $\Conj_{\phi,\psi}(F)$, in practice it is most efficient to find such an $f_0$ and then compute $\Aut_\phi(F)$ as in the next section.  


\subsection{Gr\"obner Bases}
	Buchberger's algorithm allows one to compute the points of a zero-dimensional scheme by constructing a Gr\"obner basis for its ideal of definition $I$ with respect to an appropriate monomial ordering \cite[Ch.~15]{Eisenbud_Comm_Alg}. Over a fixed polynomial ring, its performance typically degrades as the degrees of the generators of $I$ grow. When $d = \deg(\phi) = \deg(\psi)$, we saw in \S\ref{Sec: Conjugate} that $\Conj_{\phi, \psi}$ is a zero-dimensional scheme naturally defined by $2d + 1$ homogeneous polynomials of degree $d+1$ in four variables.  

	We used built-in functions in \texttt{Magma} to compute $\Aut_{\phi}(\Q)$ and compare the methods described below.  We found that this Gr\"obner basis technique can be competitive with the other algorithms developed for $d\approx 3$, but for larger $d$ this method is substantially worse.  See the tables in~\S\ref{Sec: Examples} for sample run times.


\subsection{Finite Fields --- Exhaustive Search}
	Writing $\FF_q$ for the finite field with $q$ elements, one sees that $\PGL_2(\FF_q)$ contains $q(q^2 - 1)$ elements. When $q$ and $d$ are small, it is reasonably efficient to compute $\Conj_{\phi, \psi}(\FF_q)$ by exhaustive search. Verifying the identity $\psi \circ s = s \circ \phi$ requires $O(d^2 \log^3 q)$ bit operations for a general choice of $\phi$ and $\psi$ of degree~$d$ and an element $s \in \PGL_2(\FF_q)$. Since we expect $\Conj_{\phi, \psi}(\FF_q)$ is empty, this method typically requires $O(q^3)$ such verifications to complete.  When $\phi = \psi$, so that $\Conj_{\phi, \psi}(\FF_q) = \Aut_\phi(\FF_q)$, this approach can typically be accelerated by using the classification of subgroups of $\PGL_2(\FF_q)$ \cite[Thm.~D]{PGL(2)_subgroups} to build in early termination conditions. 


\subsection{Method of Invariant Sets}
\label{Sec: Invariant Sets}
	Let $F$ be an arbitrary field, and suppose $\phi, \psi: \PP^1_F \to \PP^1_F$ are morphisms of degree at least~2. In this section we describe an algorithm to compute $\Conj_{\phi,\psi}(F)$ using linear algebra, assuming the existence of a pair of subsets $T_\phi, T_\psi \subset \PP^1(F)$ such that $s(T_\phi) = T_\psi$ for all $s \in \Conj_{\phi, \psi}(F)$. Over a given field $F$, one cannot always find such subsets, but they are easy to construct over an extension field of $F$. In particular, this method lends itself naturally to computing $\Conj_{\phi, \psi}(\alg{F})$, where $\alg{F}$ is an algebraic closure of~$F$. In fact, we will see that it also gives a field of definition $E/F$ for the absolute conjugating set, although $E$ is typically not the smallest such field.

	The degree $d = \deg(\phi) = \deg(\psi)$ is the principle measure of complexity in this algorithm. If the coefficients of $\phi$, $\psi$, and a candidate element $s \in \PGL_2(E)$ have length at most $k$ bits, then verifying the equality $s \circ \phi = \psi \circ s$ requires $O(d^3 k^2)$ bit operations in general. This can be reduced to $O(d^2 \log^3q)$ if $E$ is a finite field with $q$ elements.  The number of candidates $s \in \PGL_2(E)$ is approximately $\#T_\phi^3/6 = O(d^3)$, which can make this algorithm very inefficient if the degree is large. 
	
	We begin by explaining Algorithm~\ref{Alg: Three points algorithm}, which determines $\Conj_{\phi, \psi}(E)$ if we already have the ``invariant pair'' $T_\phi, T_\psi$. Then we give a description of how one constructs such a pair. Finally, we discuss several early termination conditions for detecting whether $\phi$ and $\psi$ fail to be conjugate. 

\subsubsection{Conjugation Sets from Invariant Pairs} \label{Sec: Three Points Algorithm} For this part of the discussion, let $E$ be any field over which $\phi$ and $\psi$ are defined. Suppose that we have two finite subsets  $T_\phi, T_\psi \subset \PP^1(E)$ satisfying the following conditions
			\begin{itemize}
				\item $\#T_\phi = \#T_\psi \geq 3$, and 
				\item $s(T_\phi) = T_\psi$ for every $s \in \Conj_{\phi, \psi}(E)$.\footnote{If 
					$\Conj_{\phi, \psi}(E) = \varnothing$, then the second condition is vacuous. }
			\end{itemize}
Let us call $\{T_\phi, T_\psi\}$ an \textbf{invariant pair} for $\phi$ and $\psi$. 

\begin{algorithm}[h]
\begin{flushleft}
Input: 
	\begin{itemize}
		\item  a nonconstant rational function $\phi \in E(z)$
		\item an invariant pair $T_\phi = \{\tau_1, \ldots, \tau_n\}$ and 
			$T_\psi = \{\eta_1, \ldots, \eta_n\}$ of $\PP^1(E)$
	\end{itemize}

Output: the set $\Conj_{\phi, \psi}(E)$ \\

\medskip

create an empty list $L$ \\

\medskip

for each triple of distinct integers $i,j,k \in \{1, \ldots, n\}$:  \\
	 \tab compute $s \in \PGL_2(E)$ by solving the system of linear equations
	\[
		s(\tau_1) = \eta_i, \quad s(\tau_2) = \eta_j, \quad s(\tau_3) = \eta_k
	\]
	\tab if  $s \circ \phi = \psi \circ s$:  append $s$ to $L$ \\

\medskip

return $L$

\end{flushleft}
\caption{--- Compute $\Conj_{\phi, \psi}(E)$ given an invariant pair in $\PP^1(E)$}
\label{Alg: Three points algorithm}
\end{algorithm}

\begin{proof}[Proof of Correctness]
Given $s \in \Conj_{\phi, \psi}(E)$, there is a triple of distinct indices $i,j,k \in \{1, \ldots, n\}$ such that  $s(\tau_1) = \eta_i$, $s(\tau_2) = \eta_j$, and $s(\tau_3) = \eta_k$. Conversely, given a triple of distinct indices $i,j,k \in \{1, \ldots, n\}$, there exists a unique element $s \in \PGL_2(E)$ such that $s(\tau_1) = \eta_i$, $s(\tau_2) = \eta_j$, and $s(\tau_3) = \eta_k$. These three equations are linear in the coefficients of~$s$. One now determines if this candidate element $s$ actually satisfies the functional equation $s \circ \phi = \psi \circ s$; if that is the case, then $s \in \Conj_{\phi,\psi}(E)$. In this way, we see that $\Conj_{\phi, \psi}(E)$ is a subset of $\mathrm{Hom}_E(T_\phi, T_\psi)$, the set of elements of $\PGL_2(E)$ mapping $T_\phi$ to $T_\psi$. 
\end{proof}

\subsubsection{Constructing an Invariant Pair}\label{Sec: Invariant Pair}	We now suppose that $\phi$ and $\psi$ are conjugate rational functions defined over a field $F$ and give a construction of sets $T_\phi$ and $T_\psi$ as in the preceding subsection. We may assume that $\deg(\phi) = \deg(\psi) = d$, since otherwise $\phi$ and $\psi$ are not conjugate. 
	
	Let $\Fix(\phi)$ be the set of fixed points of $\phi$, which has cardinality between 1 and $\deg(\phi) + 1$, inclusive. Note that any element $s \in \Conj_{\phi, \psi}(F)$ necessarily maps the fixed points of $\phi$ bijectively onto the fixed points of $\psi$. Indeed, $s$ is invertible, and if $x \in \Fix(\phi)$, then $\psi(s(x)) = s(\phi(x)) = s(x)$. Consequently, if the number of fixed points of $\phi$ differs from that of $\psi$, then $\phi$ and $\psi$ are not conjugate. A similar calculation shows that if $x \in \PP^1(F)$ is any point, then $s$ maps the set $\phi^{-n}(x)$ bijectively onto the set $\psi^{-n}(s(x))$ for each $n \geq 1$. Since $\phi$ and $\psi$ are conjugate, it is necessary that the sets $\phi^{-n}(\Fix(\phi))$ and $\psi^{-n}(\Fix(\psi))$ have the same cardinality for each $n \geq 1$.

	Define a set $T_\phi \subset \PP^1(\alg{F})$ by the following formula:
	\begin{equation}
	\label{Eq: Invariant pair}
		T_\phi = \begin{cases}
			\Fix(\phi) & \text{if } \#\Fix(\phi) \geq 3 \\
			\phi^{-1}(\Fix(\phi)) & \text{if } \#\Fix(\phi) = 2 \\
			\phi^{-2}(\Fix(\phi)) & \text{if } \#\Fix(\phi) = 1.
		\end{cases}
	\end{equation}
We claim that $T_\phi$ has cardinality at least~3 in all cases. This is evident in the first case. In the second, note that $\Fix(\phi) \subset \phi^{-1}(\Fix(\phi))$. So if $\#T_\phi = 2$, then each point of $\Fix(\phi)$ is totally ramified for $\phi$. The derivative at each of the fixed points vanishes,\footnote{More precisely, the induced map $T\phi$ on the tangent space $T\PP^1_x$ is zero.} which  means that each element of $\Fix(\phi)$ has fixed point multiplicity~1. But the total number of fixed points of a map of degree~$d$ is $d+1 \geq 3$, counting multiplicities, so we have a contradiction. (See, for example, \cite[Appx.~A]{Faber_Granville_Crelle_2011}.) Finally, suppose that we are in the third case, so that $\Fix(\phi)= \{x\}$. We claim that $\#\phi^{-1}(x) \geq 2$, for otherwise $x$ is ramified for $\phi$, which implies that the derivative $\phi'(x)$ vanishes there. But the fact that $x$ is the unique fixed point of $\phi$ means that in local coordinates centered at $x$ our map is of the form $z \mapsto z + a_{d+1}z^{d+1} + \cdots$ with $a_{d+1} \neq 0$. The derivative cannot vanish at $x$, and we must have 	$\#\phi^{-1}(x) \geq 2$ as desired. If $\phi^{-1}(x)$ consists of at least three points, then evidently so does $T_\phi = \phi^{-2}(x)$.   Otherwise, $\phi^{-1}(x) = \{x,y\}$, which means that $T_\phi = \phi^{-2}(x) = \{x,y\} \cup \phi^{-1}(y)$, which satisfies $3 \leq \#T_\phi \leq d + 2$. 

	Define $T_\phi$ as in the preceding paragraph, and define $T_\psi$ using the same recipe applied to~$\psi$. Write $E = F(T_\phi \cup T_\psi)$ for the field extension generated by the elements of $T_\phi \cup T_\psi$. Then $s(T_\phi) = T_\psi$ for every $s \in \Conj_{\phi, \psi}(E)$. We have therefore constructed an invariant pair. 

\subsubsection{Rationality Issues} The method of invariant sets produces a finite Galois extension $E/ F$ and the set $\Conj_{\phi, \psi}(E)$. One has some control over the field $E$, but it is often impossible to choose $E = F$ with this technique. An element $s \in \Conj_{\phi, \psi}(E)$ lies in $\Conj_{\phi, \psi}(F)$ if and only if it is invariant under the action of the Galois group $\Gal(E/F)$. Computing this Galois group is computationally impractical for most fields. However, when $F$ is a finite field with $q = p^r$ elements, the Galois group of $E/F$ is generated by the $q$-power Frobenius map. So it is possible to use the method of invariant sets to compute $\Conj_{\phi, \psi}(F)$ and $\Aut_\phi(F)$ for finite fields. 

	Additionally, if one is only interested in computing $\Conj_{\phi,\psi}(F)$, then it is enough to compute a subset of $\Conj_{\phi,\psi}(E)$ that satisfies conditions necessary for Galois invariance.  More specifically, we can restrict our attention to any tuples of $\mathrm{Gal}(\alg{F} / F)$-stable subsets of $T_\phi$ and $T_\psi$ as defined in \eqref{Eq: Invariant pair}, provided that the union of the subsets have at least~3 elements each.  For example, assume that $T_{\phi}$ contains exactly one $F$-rational point $y$ and exactly $2$ points $z_1,z_2$ defined over a quadratic extension $F'/F$.  Then if $\Conj_{\phi,\psi}(F)\ne\varnothing$, $T_{\psi}$ must contain points $y',z_1', z_2'$ with the same properties, and any $s\in\Conj_{\phi,\psi}(F)$ must satisfy $s(y) = y', s(z_i) \in \{z_1', z_2'\}$ for $i = 1,2$.


\subsection{Number Fields --- Chinese Remainder Theorem}
\label{Sec: Number Fields}
	We saw above that the method of invariant sets becomes impractical when the degree $d$ is large. 
Over a number field $K$, we can give an alternative algorithm that works well for large degree and has the added benefit of computing the $F$-rational points of $\Conj_{\phi, \psi}(F)$ (as opposed to the $E$-rational points for some Galois extension $E/F$ over which we have little control). We use an approach that is ubiquitous in number theory: first compute the conjugation set over a residue field $\F_v$ for some finite place(s) $v$, and then use the local information to obtain a global answer. More precisely, our method is as follows. (See also Algorithm~\ref{Alg:NumberField}.)

	Initialize empty sets $\texttt{Conjs}$ and $S_0$ and set $i := 0$.  Let $v_i$ be a finite prime outside of $S\cup S_0$, where $S$ is defined as in Corollary~\ref{Cor: Injectivity}.  Compute $\Conj_{\phi, \psi}(\F_{v_i})$ using one of the methods described above. (If $\phi = \psi$, it is typically more efficient to use the method of fixed points detailed in the next section.)  Let $G\subseteq \Aut(\PP^1_{\OO_K/\prod v_j})$ be a subset that surjects onto $\Conj_{\phi, \psi}(\F_{v_j})$ for each $j = 0, \ldots, i$; we refer to this step as the CRT (Chinese Remainder Theorem) step.  For each element $g\in G$, choose a lift $f_g\in\Aut(\PP^1_K)$ of minimal height.  If $f_g \circ \phi  = \psi \circ f_g$ then add $f_g$ to \texttt{Conjs}.  After this is complete, check if \texttt{Conjs} surjects onto $\Conj_{\phi, \psi}(\F_{v_j})$ for \emph{any} $j\in\{0,\ldots, i\}$.  If so, then $\texttt{Conjs} = \Conj_{\phi,\psi}(K)$ and we are done.  If not, then append $v_i$ to $S_0$, increment $i$, and repeat.

	In order to make this method into an algorithm, we need to provide a terminating condition.  Write $N(v)$ for the norm of a finite prime $v$. We claim that if $\prod_{i} \textup{N}(v_i) \geq 2^{[K:\Q]}M^2$, for some explicitly computable constant $M$, then \texttt{Conjs} $ = \Conj_{\phi, \psi}(K)$, even if \texttt{Conjs} does not surject onto $\Conj_{\phi, \psi}(\F_{v_j})$ for any fixed $i$.  We will spend the rest of the section proving this claim via the theory of heights.

	Let $H_K\colon \PP^1(K) \to \RR_{\ge 1}$ denote the relative multiplicative height for $K$ and let $L_2(f)$ denote the $L_2$-norm of a polynomial $f$.  See, for example, \cite[B.2, B.7]{Hindry_Silverman_Book_2000} for definitions. 

	\begin{prop}\label{prop:HeightBound}
		Let $T, T'\subset\PP^1(\alg{K})$ be Galois invariant sets of order at least $3$, and let $f_T, f_{T'}\in K[w,z]_{(0)}$ be square-free polynomials such that $V(f_T) = T$ and $V(f_{T'}) = T'$.  Then for any $s\in \Aut(\PP^1_K) \subset \PP^3(\alg{K})$ such that $s(T) = T'$, we have $H_K(s) \le 6^{[K:\Q]}L_2(f_T)^3L_2(f_{T'})^3$.
	\end{prop}
	
	\begin{proof}
		Let $s$ be as in the statement of the Proposition.  Let $\tau_1, \tau_2, \tau_3$ be $3$ distinct elements of $T$, and let $\eta_i := s(\tau_i) \in T'$.  In coordinates, we write $\tau_i = (\tau_{i,0}:\tau_{i,1})$ and $\eta_i = (\eta_{i,0}:\eta_{i,1})$.  Since an automorphism of $\PP^1$ is determined by its action on $3$ elements, we have an expression for $s = \mat{\alpha & \beta\\\gamma&\delta}$ in terms of $\tau_{i,j}, \eta_{i,j}$, i.e.
		\[
		\begin{array}{ll}
			\alpha = \sum_{\sigma\in S_3}(\textup{sgn }\sigma)
				B_{\sigma(1)}C_{\sigma(2)}D_{\sigma(3)}, &
			\beta = \sum_{\sigma\in S_3}(\textup{sgn }\sigma)
				A_{\sigma(1)}C_{\sigma(2)}D_{\sigma(3)}, \\
			\gamma = \sum_{\sigma\in S_3}(\textup{sgn }\sigma)
				A_{\sigma(1)}B_{\sigma(2)}D_{\sigma(3)}, &
			\delta = \sum_{\sigma\in S_3}(\textup{sgn }\sigma)
				A_{\sigma(1)}B_{\sigma(2)}C_{\sigma(3)},
		\end{array}
		\]
		where $A_i = \tau_{i,0}\eta_{i,1}$, $B_i = -\tau_{i,1}\eta_{i,1}$, $C_i = -\tau_{i,0}\eta_{i,0},$ and $D_i = \tau_{i,1}\eta_{i,0}.$

		This expression allows us to obtain a bound on the local height of $s$.  Let $v$ be any place of $K$ and let $\varepsilon_v = 6$ if $v \mid \infty$ and $\varepsilon = 1$ if $v \nmid \infty$.  Then, by the triangle inequality,
		\[
			|\alpha|_v \leq \varepsilon_v \cdot 
			\max_{\sigma\in S_3} |B_{\sigma(1)}C_{\sigma(2)}D_{\sigma(3)}|_v 
			\leq \varepsilon_v  \prod_{1 \leq i \leq 3} \max\{|\tau_{i0}|_v,|\tau_{i1}|_v \} \cdot
				\max\{|\eta_{i0}|_v,|\eta_{i1}|_v \}.
		\]
		One can easily check that the same bound holds for $|\beta|_v, |\gamma|_v, |\delta|_v$.
		It follows that 
		\begin{align*}
			H_K(s) &= 
				\prod_v \max\{
				|\alpha|_v, |\beta|_v, |\gamma|_v, |\delta|_v \}^{[K_v:\Q_v]} \\
				&\leq \prod_v \varepsilon_v^{[K_v:\Q_v]} \cdot 
				\prod_{1\leq i \leq 3} 
				\max\{|\tau_{i0}|_v,|\tau_{i1}|_v \}^{[K_v:\Q_v]} \cdot
				\max\{|\eta_{i0}|_v,|\eta_{i1}|_v \}^{[K_v:\Q_v]} \\
				&= 6^{[K:\Q]} \prod_{1 \leq i \leq 3} H_K(\tau_i) H_K(\eta_i).
		\end{align*}
		Since $H_K(\tau_i)\leq L_2(f_T)$ and $H_K(\eta_i) \leq L_2(f_{T'})$~\cite[Lemma B.7.3.1]{Hindry_Silverman_Book_2000}, this completes the proof.
	\end{proof}

	\begin{cor}\label{cor:height_bound}
		Let $\phi, \psi \in K(z)$ be rational functions of degree $>1$, let $T_\phi, T_\psi \subset \PP^1(\alg{K})$ be an invariant pair as in Section~$\ref{Sec: Three Points Algorithm}$ that is stable under the action of $\mathrm{Gal}(\alg{K}/K)$.\footnote{Observe that the invariant pairs constructed in \S\ref{Sec: Invariant Pair} are Galois stable.} Let $f_{T_\phi}, f_{T_\psi}$ be square-free polynomials such that $V(f_{T_\phi}) = T_\phi$ and similarly for $f_{T_\psi}$. Then every element of $\Conj_{\phi, \psi}(K) \subset \PGL_2(K) \subset \PP^3(K)$ has relative multiplicative height bounded by $6^{[K:\QQ]} L_2(f_{T_\phi})^3L_2(f_{T_\psi})^3$ .
	\end{cor}

	We take this height bound $6^{[K:\QQ]} L_2(f_{T_\phi})^3L_2(f_{T_\psi})^3$ to be our explicit constant $M$.  Now we need to show that if $\prod_{i} \textup{N}(v_i) \geq 2^{[K:\Q]}M^2$, then each element of $\Conj_{\phi, \psi}(K)$ is a lift of an element of $\prod_{i}\Conj_{\phi, \psi}(\F_{v_i})$ of minimal height.  We will need the following two lemmas.

	\begin{lemma}\label{lemma:MinHeightBound}
		Let $\mathfrak{b}$ be a nonzero fractional ideal of $\OO_K$, and write it as a quotient $\mathfrak{b} = \mathfrak{b^+} / \mathfrak{b^-}$ of relatively prime integral ideals. Then $H_K(b) \ge \textup{N}(\mathfrak{b^+})$ for all nonzero $b \in \mathfrak{b}$.
	\end{lemma}

	\begin{proof}
		Since $b \in \mathfrak{b}$, we have $|b|_v \le 1$ for any finite place $v$ such that $v(\mathfrak{b}) \geq 0$.  Therefore 
	\begin{align*}
		H_K(b) &= \prod_{v\mid\infty} \max\{1, |b|_v\}^{[K_v:\Q_v]} 
			\prod_{\substack{v\nmid \infty \\ v(\mathfrak{b}) < 0}} 
				\max\{1, |b|_v\}^{[K_v:\Q_v]}  \\
			&\geq \prod_{v\mid\infty} |b|_v^{[K_v:\Q_v]} 
				\prod_{\substack{v\nmid \infty \\ v(\mathfrak{b}) < 0}} 
				 |b|_v^{[K_v:\Q_v]} 
			= \prod_{\substack{v\nmid \infty \\ v(\mathfrak{b}) \geq 0}} |b|^{-[K_v:\Q_v]}_v,
	\end{align*}
where the last equality follows from the product formula.  Let $e_{\frakp}$ be such that $\mathfrak{b} = \prod\frakp^{e_{\frakp}}$.  Since $b \in \mathfrak{b}$, $v(b) \ge e_{\frakp_v}$, so $|b|_v^{-{[K_v:\Q_v]}}\ge \textup{N}(\frakp_v)^{e_{\frakp_v}}$.  
	\end{proof}

	\begin{lemma}\label{lem:lifts}
		Let $\fraka\subset\OO_K$ be an integral ideal, and let $\rho_{\fraka}\colon\PP^n(\OO_K)\to \PP^n(\OO_K/\fraka)$ denote the canonical projection.  For each ${b} = (b_0:b_1:\cdots:b_n) \in \PP^n(\OO_K/\fraka)$, there is at most one element ${a} = (a_0:a_1:\cdots:a_n) \in \rho_{\fraka}^{-1}({\beta})$ with $H_K({a}) <\left(2^{-[K:\QQ]}\textup{N}(\fraka)\right)^{1/2}$.
	\end{lemma}

	\begin{proof}
		Let ${a}, {a}'\in\PP^n(\OO_K)$ be such that 
			$H_K(a), H_K(a') < \left(2^{-[K:\QQ]}\textup{N}(\fraka)\right)^{1/2}$ 
	and such that $\rho_{\fraka}(a) = \rho_{\fraka}(a').$  Since $a \in \PP^n(\OO_K)$, there exists a coordinate $i_0$ such that $a_{i_0} \not\in \fraka$. It follows that $a_{i_0}' \not\in \fraka$ too. 

	Then for each $i$ and each place $v$, an argument as in Proposition~\ref{prop:HeightBound} shows that
	\[
			\max \left\{1, \left|\frac{a_i}{a_{i_0}} 
				- \frac{a_i'}{a_{i_0}'}\right|_v  \right\} 
				\leq  2 \ 
				\max_{\ell}\left\{ \left| \frac{a_\ell}{a_{i_0}}\right|_v\right\} \cdot
				\max_\ell \left\{ \left| \frac{a_\ell'}{a_{i_0'}}\right|_v\right\} .
	\]
Taking the product over all $v$ gives 
$H_K\left( \frac{a_i}{a_{i_0}} - \frac{a_i'}{a_{i_0}'} \right) \leq 2^{[K:\QQ]} H_K(a) H_K(a')$. The latter is less than $\textup{N}(\fraka)$ by hypothesis, and $\frac{a_i}{a_{i_0}} - \frac{a_i'}{a_{i_0}'}$ lies in the fractional ideal $(a_{i_0}a_{i_0}')^{-1}\fraka$, so the preceding lemma implies that $\frac{a_i}{a_{i_0}} = \frac{a_i'}{a_{i_0}'}$.  That is, $a = a'$. 
	\end{proof}

	\begin{prop}
		Let $v_0, \ldots, v_n$ be finite places of $K$ such that
		\begin{enumerate}
			\item $\phi$ and $\psi$ have good reduction at $v_i$ for all $i$; 
			\item the reduction map $\Conj_{\phi, \psi}(K) \to \Conj_{\phi, \psi}(\F_{v_i})$ 
				is injective for all $i$; and
			\item $\prod_i\textup{N}(v_i)  \geq 2^{[K:\QQ]}M^2$, where 
				$M = 6^{[K:\QQ]} L_2(f_{T_\phi})^3L_2(f_{T_\psi})^3$ 
				as in Corollary~$\ref{cor:height_bound}$.
		\end{enumerate}
		For any tuple $(g_i)\in \prod_i\Conj_{\phi, \psi}(\F_{v_i})$, let $g_K \in \Aut({\PP^1_K})$ be a simultaneous lift of each $g_i$ of minimal height.  If  $(g_i) \in\textup{im}\left(\Conj_{\phi, \psi}(K) \to \prod_i \Conj_{\phi, \psi}(\F_{v_i})\right)$, then $g_K \in\Conj_{\phi, \psi}(K)$.
	\end{prop}

	\begin{proof}
		Assume that $(g_i) \in\textup{im}\left(\Conj_{\phi, \psi}(K) \to \prod_i \Conj_{\phi, \psi}(\F_{v_i})\right)$ and let $g'\in\Conj_{\phi, \psi}(K)$ denote its pre-image.  (The element $g'$ is unique by assumption $(2)$.)  By Corollary~\ref{cor:height_bound}, $H_K(g') \leq M \leq \left(2^{-[K:\QQ]}\prod_i\textup{N}(v_i)\right)^{1/2}$.  By Lemma~\ref{lem:lifts}, $g'$ must have minimal height among all lifts, so $g' = g_K \in \Conj_{\phi, \psi}(K)$.
	\end{proof}

	\begin{algorithm}[h]
		\begin{flushleft}
			Input: a number field $K$ and rational functions $\phi, \psi \in K(z)$ of degree $d > 1$

			Output: the set $\Conj_{\phi,\psi}(K)$ \\
			\medskip

			choose an invariant pair $T_\phi, T_\psi$ as in Section~\ref{Sec: Invariant Pair} and set
				$M = 6^{[K:\Q]}L_2(f_{T_\phi})^3L_2(f_{T_\psi})^3$ \\

			\medskip

			create an empty list $L$, and 
				set $\fraka = \langle 1\rangle$\\
			for $v$ a prime of good reduction at $v$ such that
				$\Conj_{\phi,\psi}(K)\to\Conj_{\phi,\psi}(\FF_v)$ is injective:\\
				\tab compute $\Conj_{\phi, \psi}(\FF_v)$  \\ 
				\tab if $\Conj_{\phi, \psi}(\FF_v) = \varnothing$:\\
				\tab \tab return $\varnothing$\\
				\tab else:\\
				\tab \tab append $\Conj_{\phi, \psi}(\FF_v)$ to L, 
					 and set $\fraka = \fraka\frakp_v$\\
				\tab Set $L' = CRT(L)$ and initialize an empty list \texttt{Conjs}\\
				\tab for $s$ in $L'$:\\
				\tab \tab set $s'\in \textup{PGL}_2(\OO_K)$ to be a lift of $s$ of minimal height\\
				\tab \tab if $H_K(s') \le M$ and $s' \circ \phi = \psi \circ s' $:\\
				\tab \tab \tab append $s'$ to \texttt{Conjs}\\
				\tab if $\textup{N}(\fraka) \geq 2^{[K:\QQ]} M^2$ or if $\#\text{\texttt{Conjs}} = \#\Conj_{\phi, \psi}(\F_{v})$ for any $v\mid\fraka$:\\
				\tab \tab return \texttt{Conjs}

		\end{flushleft}

		\caption{--- Computation of $\Conj_{\phi,\psi}(K)$ via the Chinese Remainder Theorem}
		\label{Alg:NumberField}
	\end{algorithm}

		There are a few technical details that we have left out in our description of Algorithm~\ref{Alg:NumberField} in the case that $\phi = \psi$, specifically in  where we decide whether to terminate and in the Chinese Remainder Theorem step.  These details allow us to avoid extraneous computation.  We give an example here, and the curious reader can find the rest in our source code.

	It is possible for the reduction of $\Aut_{\phi}(K)$ to be a proper subgroup of $\Aut_{\phi}(\F_v)$ for all places $v$ of good reduction.  Consider the rational function $\phi(z) = 2z^5$.  One can use the method of invariant sets to check that
		\[ 
			\Aut_{\phi}(\overline{\Q}) = \left\{z, iz, -z, -iz, (\sqrt2z)^{-1}, i(\sqrt2z)^{-1}, 
			-(\sqrt2z)^{-1}, -i(\sqrt2z)^{-1}\right\}, 
		\]
which is a dihedral group of order~8. For all primes $p>2$, at least one of $-1, 2, -2$ is a square in $\F_p$.  Therefore, $\Aut_{\phi}(\F_p)$ always contains $\Z/2\times\Z/2$ or $\Z/4$ as a subgroup.  As the algorithm is stated, we would compute a lift of every element in $\prod_{p = 5}^{19} \Aut_{\phi}(\F_p)$.  However, by $p = 7$ one can already recognize that $\Aut_{\phi}(\Q) \subseteq \Z/2$ since $\Aut_{\phi}(\F_5) = \Z/4$ and $\Aut_{\phi}(\F_7) = \Z/2 \times\Z/2$.  Our code checks for group-theoretic properties like this when deciding whether to terminate.

		When computing $\Conj_{\phi, \psi}(K)$, it is important to build in as many early termination conditions as possible, since typically the elements of $\Conj_{\phi, \psi}(K)$ have significantly smaller height than the theoretical bound~$M$.  This is of course true when $\Conj_{\phi, \psi}(K)$ is trivial, but it remains true even in the nontrivial case.  For example, consider the functions in the last line of Table~\ref{table:nontrivial_auts}.  The height bound for $\phi(z) = 345025251z^6$ is over $50$ digits, while, in contrast, the height of the non-trivial automorphism is $2601$.  The same phenomenon can be seen with many of the examples in Table~\ref{table:ConjugatesOfPowerMaps}.


\subsection{An Early Termination Criterion} \label{sec: early termination}
In order to avoid extraneous computation, we want to detect as quickly as possible when two rational functions are \textit{not} conjugate. The method of invariant sets suggests a useful criterion. 

	Let $a \in F \smallsetminus \{0\}$, let $f_1, \ldots, f_r \in F[X,Y]$ be pairwise non-commensurate irreducible homogeneous polynomials, and let $e_1, \ldots, e_r \geq 1$ be integers. We define the \textbf{factorization type} (or simply \textbf{type}) of the polynomial $f:= af_1^{e_1} \cdots f_r^{e_r}$ to be the multiset of pairs $\left\{(\deg(f_1), e_1), \cdots, (\deg(f_r), e_r) \right\}$. Note that the degree of $f$ is determined by its type. The definition of type extends in the obvious way to inhomogeneous univariate polynomials.
	
	Now suppose that $\phi, \psi \in F(z)$ are rational functions of degree $d \geq 2$ such that $\Conj_{\phi, \psi}(F)$ is nonempty. We saw in \S\ref{Sec: Invariant Pair} that for each $s \in \Conj_{\phi, \psi}(F)$, we have $s(\Fix(\phi)) = \Fix(\psi)$. In fact, more is true.  Write $\phi$ and $\psi$ in homogeneous form as $\Phi = (\Phi_0(X,Y), \Phi_1(X,Y))$ and $\Psi = (\Psi_0(X,Y), \Psi_1(X,Y))$. The polynomials $f_\phi = X\Phi_1  - Y\Phi_0$ and $f_\psi = X\Psi_1 - Y\Psi_0$ determine the fixed points of $\phi$ and $\psi$, respectively. Writing $s$ in homogeneous form as $S = (S_0(X,Y), S_1(X,Y))=(\alpha X + \beta Y, \gamma X + \delta Y)$, the condition $s \circ \phi \circ s^{-1} = \psi$ may be translated as $S \circ \Phi = \lambda \cdot \Psi \circ S$ for some $\lambda \in F^\times$. We now see that
	\begin{equation}
	\label{Eq: Move the fixed points}
		\begin{aligned}
			\lambda f_{\psi}(S_0, S_1) &= \lambda \left[S_0 \cdot \left(\Psi_1 \circ S\right) - S_1 \cdot \left(\Psi_0 \circ S\right)\right] \\
				&= S_0 \cdot \left( \gamma \Phi_0 + \delta \Phi_1\right) - S_1 \cdot \left( \alpha \Phi_0 + \beta \Phi_1\right) \\
				&= (\alpha\delta - \beta\gamma) \left(X \Phi_1 - Y \Phi_0\right) = (\alpha\delta - \beta\gamma) f_\phi.
		\end{aligned}
	\end{equation}
Hence the types of $f_\phi$ and $f_\psi$ agree. Said another way, if the types of the polynomials $f_\phi$ and $f_\psi$ do not match, then $\Conj_{\phi, \psi}(F)$ is empty. (Since $s(\phi^{-n}(\Fix(\phi))) = \psi^{-n}(\Fix(\psi))$ for every $n \geq 1$, a similar statement holds for the polynomials defining the $n{\tth}$ preimages of the fixed points.) 	
		
	Assume now that $F = \FF_q$ is the finite field with $q$ elements. By definition, the type of a homogeneous polynomial $f \in \FF_q[X,Y]$ is computed by factoring it completely. However, there are well known ``folk methods'' for calculating the type of $f$. Using only formal derivatives and the Euclidean algorithm, one can determine the number of irreducible factors of a given degree and the exponents to which they occur in $f$. (See \cite[\S2]{Cantor-Zassenhaus}.)

	If $F = K$ is a number field, then factoring $f_\phi$ and $f_\psi$ may not be computationally efficient. An alternative approach is suggested by the Chinese Remainder Theorem method for computing $\Conj_{\phi, \psi}(K)$. Let $v$ be a non-Archimedean place of $K$ at which both $\phi$ and $\psi$ have good reduction. Then each element of $\Conj_{\phi, \psi}(K)$ has good reduction at $v$, and we may reduce equation~\eqref{Eq: Move the fixed points} modulo~$v$ to obtain a relation between the fixed point polynomials of $\phi_v$ and $\psi_v$, the corresponding rational functions defined over the residue field $\FF_v$. If $\Conj_{\phi, \psi}(K)$ is nonempty, then for each place of good reduction $v$ for $\phi$ and $\psi$, the types of the polynomials $f_\phi$ and $f_\psi$ must agree modulo $v$.\footnote{When $f_\phi$ and $f_\psi$ are irreducible, it is equivalent to say that the splitting fields of $f_\phi$ and $f_\psi$ have the same Dedekind zeta function \cite{Perlis_Stuart_1995}. One says that these splitting fields are ``arithmetically equivalent.''} Algorithm~\ref{Alg:NumberField} provides a collection of places $v$ that are sufficient to compute the full set $\Conj_{\phi, \psi}(K)$ via the Chinese Remainder Theorem; one could use this set of places $v$ for our early termination criterion as well. 
	
\section{Algorithms for computing automorphisms}\label{Sec: AutAlgorithms}

	If $s$ is a non-trivial automorphism of $\phi$, then the action of $\phi$ on the fixed points of $s$ is highly restricted.  We exploit this restriction to give faster algorithms for computing the automorphism group.

\subsection{Method of Fixed Points}

	Let $F$ be a field and let $\phi: \PP^1_{F} \to \PP^1_{F}$ a nonconstant morphism. We assume that either $F$ is finite, or $\textup{char}(F)\nmid d^3 - d$.  For any $\phi$-periodic point $x \in \PP^1(\alg F)$, write $\per(x)$ for its exact period --- i.e., the minimum positive integer $i$ such that $\phi^i(x) = x$. If $x$ is not periodic,  write $\per(x) = +\infty$.  For each pair of integers $i, j \in \{1,2\}$, define the following set:
	\begin{equation}
	\label{Eq: Various sets}
		Z_{i,j} = \{x \in \PP^1(\alg F) : \per(x) = i, \ [F(x): F] = j\}.
	\end{equation}
We also define the following set of ordered pairs:
	\begin{equation}
	\label{Eq: One more set}
		W = \{(x,y) : x \in Z_{1,1}, \ y \in \phi^{-1}(x), \ [F(y): F] = 1\}.
	\end{equation}
(More concretely, $W$ is the set of pairs of $F$-rational points such that $x$ is fixed by $\phi$ and $\phi(y) = x$.) These sets may be constructed by factoring the polynomials that define the fixed points of $\phi$, the points of period~2, and the preimages of $F$-rational points. We write $Z^{(2)}$ for the set of unordered pairs of elements of a set~$Z$. 

	Let $s = \mat{\alpha & \beta \\ \gamma & \delta}$ be a nontrivial element of $\Aut_{\phi}(F)$. The homogeneous polynomial defining the fixed points of $s$ is $\gamma X^2 + (\delta - \alpha)XY - \beta Y^2$.  	Suppose that $s \in \Aut_\phi(F)$ has precisely two distinct fixed points $x_1$ and~$x_2$. Then $s(\phi(x_1)) = \phi(s(x_1)) = \phi(x_1)$, so that $\phi(x_1) \in \{x_1,x_2\}$. There are three possible cases: 
	\begin{enumerate}
		\item $\phi$ fixes both $x_1$ and $x_2$; 
		\item $\phi$ swaps $x_1$ and $x_2$; or 
		\item $\phi(x_1) = x_2$ and $\phi$ fixes $x_2$ (perhaps after interchanging $x_1$ and $x_2$). 
	\end{enumerate}
	Since $\phi$ is defined over $F$, all Galois conjugates of a fixed point must also be fixed points.  Thus in cases (1) and (2), either  $x_1$ and $x_2$ are both $F$-rational, or  they are quadratic conjugates over $F$. In case (3), both $x_1$ and $x_2$ must be $F$-rational. 

	If $x_1$ and $x_2$ are both $F$-rational in case (1) --- so that $(x_1,x_2) \in Z_{1,1}^{(2)}$ --- then we may select $u \in \PGL_2(F)$ such that $u(x_1) = \infty$ and $u(x_2) = 0$. Then $usu^{-1} = \mat{ \zeta & \\ & 1}$ for some root of unity $\zeta \in F$.  If $\zeta$ has order~$n$, then $n$ divides one of $d$, $d+1$, or $d-1$ by Proposition~\ref{Prop: Geometric Bound}.  Let $T$ be the set of roots of unity in $F$ that have order dividing $d$, $d+1$ or $d-1$. We loop over all distinct unordered pairs of elements $(x_1,x_2) \in Z_{1,1}^{(2)}$, and check which elements of $u^{-1} \mat{T & \\ & 1} u$ lie in $\Aut_\phi(F)$. See the first for--loop of Algorithm~\ref{Alg: Fixed Point Method}. In fact, this strategy works in case (2) when $x_1, x_2$ are both $F$-rational, and in case (3). These correspond to looping over pairs $(x_1,x_2)$ in $Z_{2,1}^{(2)}$ and in $W$, respectively.

	Now suppose $x_1$ and $x_2$ are quadratic Galois conjugates  in case (1), so that $(x_1,x_2) \in Z_{1,2}^{(2)}$.  We will use the following lemma.
	
	\begin{lemma}
		Let $(z_1,z_2)\in Z_{1,2}^{(2)}$.  There exists an order $n$ element $s\in\Aut_{\PP^1}(F)$ such that $\Fix(s) = \{z_1,z_2\}$ if and only if either $n = 2$ and $\mathrm{char}(F) \neq 2$ or else $F(\mu_n) = F(z_1,z_2)$.
	\end{lemma}
	\begin{proof}
		Let $u := \begin{pmatrix} 1 & - z_1\\ 1 & -z_2\end{pmatrix}$.  Then there exists such an $s$ if and only if
		\begin{equation}\label{eq:DefnOfs}
			s = u^{-1}\begin{pmatrix} \xi & \\ & 1\end{pmatrix} u 
			= \begin{pmatrix}
				z_1 - \xi z_2 & (\xi - 1)z_1z_2 \\
				1 - \xi & \xi z_1 - z_2
			\end{pmatrix},
		\end{equation}
		for some primitive $n$-th root of unity $\xi \in F(z_1,z_2)$.  The element $s$ is defined over $F$ if and only if the non-trivial element $\sigma$ of $\Gal(F(z_1,z_2)/F)$ fixes $\frac{z_1 - \xi z_2}{1 - \xi}$ and $\frac{z_2 - \xi z_1}{1 - \xi}$.  By expanding the resulting equations and noting that $\sigma(z_1) = z_2$, we see that this happens if and only if $\xi\xi^{\sigma} = 1$, which completes the proof.
	\end{proof}

	Using the lemma,  we can detect $s$ as follows.  Let $\Lambda$ be the set of $\xi\in\overline{F}$ such that $\xi$ is a root of a quadratic factor of $C_i(X) :=  X^{d + i} - 1$ for $i\in\{-1,0,1\}$, or $\xi = -1$.  Loop over all Galois conjugate pairs $\{x_1,x_2\}\in Z_{1,2}^{(2)}$ and check which elements $\xi\in\Lambda$ generate the same field extension as $x_1$ and $x_2.$  (Since everything is quadratic, it is enough to check that the quotient of the discriminants of $x_i$ and $\xi$ is a square.)  For those $\xi$, we additionally check whether $u^{-1}\mat{\xi & \\ & 1} u$ is an automorphism of $\phi$. See the second for--loop of Algorithm~\ref{Alg: Fixed Point Method}. The same strategy also applies if we are in case (2) and $x_1$ and $x_2$ are quadratic conjugates. 

	Now assume that $s$ has a unique fixed point $x$.  Then $F$ is a field of characteristic $p>0$ and $s$ has order $p$.  (Move the unique fixed point to infinity.  Then $s$ is a nontrivial translation with finite order.)  Proposition~\ref{Prop: Geometric Bound} implies that $p = \textup{ord}(s) | d^3 - d$, which, together with our assumptions on $F$, forces $F$ to be finite.  Since $F$ is perfect, $x$ must be $F$-rational.  Now
	\[
	s(\phi(x)) = \phi(s(x)) = \phi(x), \text{ so that }\phi(x) = x.
	\]
	 Hence $x \in Z_{1,1}$. Choose $u \in \PGL_2(F)$ such that $u(x) = \infty$; then $u s u^{-1} = \mat{1 & \lambda \\ & 1}$ for some $\lambda \in F\smallsetminus\{0\}$. That is, $s \in u^{-1} \mat{1 & F\smallsetminus\{0\} \\ & 1}u$.  In order to find all elements of $\Aut_\phi(F)$ of order $p$, it suffices to apply this technique to every $x$ in the set $Z_{1,1}$.  See the last for--loop of Algorithm~\ref{Alg: Fixed Point Method}.  
	 
	\begin{remark}
		The final step is the only one where we use the assumption that $F$ is not infinite of characteristic $p$, where $p|d^3 - d$.   If $F$ is infinite of characteristic $p$ and $p|d^3 - d$, then one can simply omit the final step, and the algorithm will return the subset of $\Aut_{\phi}(F)$ consisting of all automorphisms with order different from $p$.  See also \S\ref{Sec: Large Fields of char p}.
	\end{remark}
	
	\begin{remark}
		This algorithm does not rely on the classification of finite subgroups of $\textup{PGL}_2$, so in principle it could be generalized to higher dimensions.  However, there are some practical difficulties to overcome; for example, the naive generalization would quickly become too cumbersome combinatorially due to the many possible ways that a morphism $\phi: \PP^n \to \PP^n$ could act on the $n+1$ fixed points of an automorphism.  It would be interesting to find an elegant way of controlling this combinatorial explosion and the other difficulties that arise.
	\end{remark}	
	
	\begin{algorithm}[h]
	\begin{flushleft}
	Input:  a field $F$ and $\phi \in F(z)$ of degree $\geq 2$ \\

	Output: the set $\Aut_\phi(F)$, if $F$ is finite or if $\textup{char}(F)\nmid d^3 - d$; \\
	\tab\tab\tab\tab the set $\{s\in\Aut_{\phi}(F): \ord(s) \ne p\}$, otherwise

	\medskip

	let $C_i(X) = X^{d + i} -1$ for $i = -1, 0, 1$\\
	create a list $T$ of $F$-rational roots of $C_i(X)$ \\
	create a list $\Lambda$ of roots of $F$-quadratic factors of $C_i(X)$ and $-1$\\

	\medskip

	create a list $L = [ z  ]$ \\
	create the sets $Z_{i,j}, W$ defined in equations~\eqref{Eq: Various sets} and~\eqref{Eq: One more set}
	\medskip

	for each pair $(x,y)$ with $x \neq y$ in $Z_{1,1}^{(2)} \cup Z_{2,1}^{(2)} \cup W$: \\
		\tab choose $u \in \PGL_2(F)$ such that $u(x) = \infty$ and $u(y) = 0$ \\
		\tab for $\zeta \in T \smallsetminus \{1\}$: \\
		\tab \tab set $s(z) = u^{-1}(\zeta u(z) )$ \\
		\tab \tab if $s \circ \phi = \phi \circ s$: append $s$ to $L$ \\

	\medskip

	for each pair of Galois conjugates $(x,y)$ in $Z_{1,2}^{(2)} \cup Z_{2,2}^{(2)}$: \\
		\tab choose $u \in \PGL_2(F(x,y))$ such that $u(x) = \infty$ and $u(y) = 0$ \\
		\tab for $\xi \in \Lambda$: \\
		\tab \tab set $s(z) = u^{-1}(\xi u(z) )$ \\
		\tab \tab if $s \circ \phi = \phi \circ s$: append $s$ to $L$ \\

	\medskip

	if $p := \textup{char}(F)|(d^3 - d)$ and $F$ is finite:\\
		\tab for $\lambda$ in $F^{\times}/\F_p^{\times}$:\\
		\tab \tab for $x \in Z_{1,1}$: \\
		\tab \tab \tab choose $u \in \PGL_2(F)$ such that $u(x) = \infty$ \\
		\tab \tab \tab set $s(z) = u^{-1} (u(z) + \lambda)$ \\
		\tab \tab \tab if $s \circ \phi = \phi \circ s$: append $s^i$ to $L$ for all $1\leq i< p$\\

	\medskip

	return $L$
	\end{flushleft}

	\caption{--- Computation of $\Aut_\phi(F)$ via the method of fixed points}
	\label{Alg: Fixed Point Method}
	\end{algorithm}

\subsection{Finding order $p$ automorphisms for large fields of 
characteristic $p$}\label{Sec: Large Fields of char p}

	Let $F$ be a field of characteristic $p$, and $\phi\in F(z)$ of degree $d>2$.  If $p|d^3 - d$ and $F$ is infinite, the last loop in Algorithm~\ref{Alg: Fixed Point Method} does not terminate; if $F$ is finite, but has cardinality much larger than the degree, then the last loop may be the dominant step.  In those cases, we can compute the order $p$ elements by a hybrid of the method of fixed points and the method of invariant sets.
	
	As we saw in the previous section, if $s\in\Aut_{\phi}(F)$ has order $p$, then $s$ has a unique fixed point $x$ which is either $F$-rational or generates an inseparable quadratic extension of $F$.  Since $s\circ\phi = \phi\circ s$, $x$ must also be a fixed point of $\phi$.  In addition, since $s$ permutes the fixed points of $\phi$, $\Fix(\phi)\smallsetminus\{x\}$ must break up into disjoint orbits of size $p$.  In particular $\#\Fix(\phi)\equiv 1\pmod{p}$.
	
	First consider the case when $x$ is $F$-rational; let $u\in\PGL_2(F)$ be such that $u(x) = \infty$, then $usu^{-1} = \mat{1 & \lambda\\0 & 1}$, where $\lambda \in F$.  If $\#\Fix(\phi) > 1$, let $y_1 \in \Fix(\phi)\smallsetminus \{x\}$, and set $y_2 := s(y_1)$.  Then $u(y_2) = u(y_1) + \lambda$, and in particular, $u(y_2) - u(y_1)\in F$.  So we may detect $s$ by looping over all $F$-rational fixed points $x$, choosing $y_1 \in\Fix(\phi)\smallsetminus\{x\}$, looping over all $y_2\in\Fix(\phi)(F(y_1))\smallsetminus \{x, y_1\}$ and testing (either by using a basis representation of $F(y_1)/F$ or by Galois descent) whether the element $u^{-1}\mat{1 & u(y_2) - u(y_1)\\0& 1}u\in \Aut_{\phi}(F)$.  If $\#\Fix(\phi) = 1$, then we may use the same argument with $y_1, y_2 \in \phi^{-1}(x) \smallsetminus \{x\}$ (by arguments in ~\S\ref{Sec: Invariant Sets}, $\#\phi^{-1}(x) > 1$). See Algorithm~\ref{Alg: Order p automorphisms}. 
	
	Now consider the case when $x$ generates an inseparable quadratic extension of $F$.  Then $F$ has characteristic $2$ and  $s = \mat{ \mu &x^2 \\ 1 & \mu}$, where $\mu\in F$.  If $\#\Fix(\phi)> 1$, then let $y_1\in\Fix(\phi)\smallsetminus\{x\}$ and set $y_2:= s(y_1)$; note that this implies that $y_2\in F(y_1)$.  Then we can solve for $\mu$ in terms of $y_1$ and $y_2$:
	\[
		\mu = 
			\begin{cases}
				y_2 & \textup{if } y_1 = \infty \\
				y_1 & \textup{if } y_2 = \infty \\
				\displaystyle \frac{y_1y_2 + x^2}{y_1 + y_2} & \textup{if } y_1, y_2 \neq \infty.
			\end{cases}
	\]
If $\#\Fix(\phi) = 1$, then we may use the same argument with $y_1, y_2 \in \phi^{-1}(x) \smallsetminus \{x\}$ instead.

	\begin{algorithm}[h]
	\begin{flushleft}
	Input:  a finite field $F$ of characteristic $p$ and $\phi \in F(z)$ of degree $d\geq 2$ \\

	Output: the subset of $\Aut_\phi(F)$ consisting of all order $p$ elements

	\medskip
	
	if $p\nmid d^3 - d$ or $\#\Fix(\phi)\not\equiv 1 \pmod{p}$ or ($\#\Fix(\phi) = 1$ and $\#\phi^{-1}(\Fix(\phi)) \not\equiv 1 \pmod{p}$):\\
	\tab return $\varnothing$\\
	
	\medskip
	create an empty list $L$ \\
	
	\medskip
	if $\# \Fix(\phi) > 1$: set $T = \Fix(\phi)$ \\
	else: set $T = \phi^{-1}(\Fix(\phi))$ \\
	
	\medskip
	
	for $x \in \Fix(\phi)(F)$:\\
	\tab choose $u\in\PGL_2(F)$ such that $u(x) = \infty$\\
	\tab choose $y_1\in T\smallsetminus\{x\}$\\
	\tab for $y_2\in T \smallsetminus\{x,y_1\}$, $y_2\in F(y_1)$:\\
	\tab\tab if $u(y_2) - u(y_1)\in F$ and $s(z):=u^{-1}(u(z) + u(y_2) - u(y_1))$ satisfies $s\circ\phi = \phi\circ s$:\\
	\tab\tab\tab append $s$ to $L$\\
	
	\medskip
	
	for $x\in\Fix(\phi)$ such that $F(x)/F$ is purely inseparable of degree $2$:\\
	\tab choose $y_1\in T \smallsetminus\{x\}$\\
	\tab for $y_2\in T \smallsetminus\{x,y_1\}$, $y_2\in F(y_1)$:\\
	\tab\tab set $s(z) := (\mu z + x^2)/(z + \mu)$, with $\mu$ as in \S\ref{Sec: Large Fields of char p}\\
	\tab\tab if $s\circ\phi = \phi\circ s$:\\
	\tab\tab\tab append $s$ to $L$\\
	
	\medskip

	return $L$
	\end{flushleft}

	\caption{--- Computation of the elements of order $p$ in $\Aut_\phi(F)$ 
	}
	\label{Alg: Order p automorphisms}
	\end{algorithm}

	\section{Examples}\label{Sec: Examples}

		In this section, we compute some examples to give an idea of the running times of the different algorithms over $\QQ$.  Since the method of fixed points from~\S\ref{Sec: AutAlgorithms} can only be applied to the computation of $\Aut_{\phi}(\QQ)$ and not $\Conj_{\phi,\psi}(\QQ)$, we restrict to computing automorphisms for comparison purposes. We write CRT, FP, and GB for the Chinese Remainder Theorem, fixed points, and Gr\"obner basis methods of computing $\Aut_\phi(\QQ)$, respectively. 
		
		First, we present some hand-selected examples with nontrivial automorphism group which demonstrate the correctness of the algorithm (Table~\ref{table:nontrivial_auts}).  As an approximation of ``random'' rational functions with non-trivial automorphism group, we compute the automorphism group of conjugates of $z^k$, where the conjugating functions were chosen randomly (Table~\ref{table:ConjugatesOfPowerMaps}).  Then we present median running times for randomly generated rational maps of varying degrees and varying heights (Table~\ref{table:fixed_points}).  In this last table, we did not include the running times of the Gr\"obner basis method when $d > 9$ since it is already apparent that this method was no longer competitive.  All of the randomly generated functions had trivial automorphism group.

		Our computations indicate that the fixed point method is faster for random rational functions of small degree, but that the CRT method is a better choice once the degree is larger than $12$.  In our implementation, the main bottleneck in the fixed point algorithm is in computing $Z_{1,2}$ and $Z_{2,2}$; this requires computing the quadratic factors of a degree $d^2 + 1$ polynomial. Implementing a faster method for finding quadratic factors of large degree polynomials  may render the fixed point method feasible for larger degrees.  As mentioned in~\S\ref{Sec: AutAlgorithms}, the CRT method becomes slower if there exists an automorphism with large height.  Thus, if one suspects that there will be a non-trivial automorphism, then it may be preferable to use the fixed point method even if the degree is large.  Interestingly, the height of the rational function seems to have little effect on the running times of the fixed point method and the CRT method, in stark contrast to the Gr\"obner basis method (Table~\ref{table:fixed_points}).
		
		These examples were computed on a Macbook Air (Apple, Inc.) running Mac OS X 10.7.2 with a 2.13 GHz Intel Core 2 Duo processor and 2GB of RAM. The fixed point method and CRT method were run with Sage 4.7.2 which was released on October 29, 2011.  The Gr\"obner basis method was run with \texttt{Magma} V2.17-1.  It is possible that the running time gap between our algorithms and the ``naive'' Gr\"obner basis algorithm is partly due to this difference is programs; however, the gap is so large that we believe it cannot possibly account for all of the difference. 
		
		All running times are listed in seconds.
		
		\vspace{.6in}
		
		{\small	\begin{table}[!h]\renewcommand{\arraystretch}{1.5}
			\begin{tabular}{|c|c|c|c|c|c|}
				\hline
				$\phi$ &  CRT & FP & GB & $\Aut_{\phi}(\Q)$ & group \\
				\hline
				$\frac{z^2 - 2z - 2}{-2z^2 - 2z + 1}$
					& 0.35
					& 0.05
					& 0.04
					& $z^{\pm1}, \left(\frac{-z}{z + 1}\right)^{\pm1}, (-z - 1)^{\pm1}$ 
					& $\mathfrak{D}_6$ \\
				\hline
				$\frac{z^2 + 2z}{-2z - 1}$ 
					& 0.09
					& 0.03
					& 0.01 
					& $z, \frac{-z}{z + 1}, \frac{1}{z},-z - 1, \frac{-z - 1}z, 
						\frac{-1}{z + 1}$ 
					& $\mathfrak{D}_6$ \\
				\hline
				$\frac{z^2 - 4z - 3}{-3z^2 - 2z + 2}$ 
					& 0.07
					& 0.02
					& 0.02
					& $z, \frac{-z - 1}z, \frac{-1}{z + 1}$
					& $\mathfrak{C}_3$ \\
				\hline
				\multirow{2}{*}{$\frac{z^5 + 5z^4 - 20z^3 + 10z^2 + 5z - 2}{
					2z^5 - 5z^4 - 10z^3 + 20z^2 - 5z - 1}$} 
					& \multirow{2}{*}{0.49} 
					& \multirow{2}{*}{0.06}
					& \multirow{2}{*}{0.13}
					& $z, -z + 1, \frac1z, \frac{z}{z - 1},\frac{2z - 1}{z - 2},
					\frac{-z + 2}{z + 1},$
					& \multirow{2}{*}{$\mathfrak{D}_{12}$} \\
					& & & & $\frac{z + 1}{2z - 1},\frac{z - 2}{2z - 1}, 
					\frac{-1}{z - 1}, 
					\frac{z - 1}z,\frac{-z - 1}{z - 2}, \frac{2z - 1}{z + 1}$ 
					& \\
				\hline
				\multirow{2}{*}
				{$\frac{z^5 - 5z^4 + 10z^2 - 5z}{-5z^4 + 10z^3 - 5z + 1}$} 
					& \multirow{2}{*}{2.36} 
					& \multirow{2}{*}{0.15}
					& \multirow{2}{*}{0.15} 					
					& $z, \frac{z}{z - 1}, -z + 1, \frac{1}z, \frac{2z - 1}{z - 2}, 
					\frac{-z + 2}{z + 1},$
					& \multirow{2}{*}{$\mathfrak{D}_{12}$} \\
					& & & &  $\frac{z - 2}{2z - 1}, \frac{z + 1}{2z - 1}, 
					\frac{-1}{z - 1}, 
					\frac{z - 1}z, \frac{-z - 1}{z - 2}, \frac{2z - 1}{z + 1}$
					& \\
				\hline
				$\frac{z^5 - 20z^4 + 30z^3 + 10z^2 - 20z + 3}{
					-3z^5 - 5z^4 + 40z^3 - 30z^2 - 5z + 4}$ 
					& 0.63
					& 0.03
					& 0.16
					& $z,\frac{z - 2}{2z - 1}, \frac{-1}{z - 1},\frac{z - 1}z, 
					\frac{-z - 1}{z - 2},\frac{2z - 1}{z + 1}$
					& $\mathfrak{C}_6$\\
				\hline
				$\frac{3z^2 - 1}{z^3 - 3z}$ 
					& 0.20
					& 0.04
					& 0.02
					& $\pm z, \pm\frac1z, 
					\pm\left(\frac{-z + 1}{z + 1}\right), 
					\pm\left(\frac{z + 1}{z - 1}\right)$
					& $\mathfrak{D}_8$ \\
				\hline
				$\frac{z^3 - 3z}{-3z^2 + 1}$ 
					& 0.23
					& 0.03
					& 0.02
					& $\pm z, \pm\frac1z, 
					\pm\left(\frac{-z + 1}{z + 1}\right), 
					\pm\left(\frac{z + 1}{z - 1}\right)$
					& $\mathfrak{D}_8$ \\
				\hline
				$\frac{z^3 - 21z^2 - 3z + 7}{-7z^3 - 3z^2 + 21z + 1}$ 
					& 0.38
					& 0.02
					& 0.03
					& $z, \frac{-1}z,\frac{z - 1}{z + 1},\frac{-z - 1}{z - 1}$
					& $\mathfrak{C}_4$\\
				\hline
				$\frac{z^{11} + 66z^6 - 11z}{-11z^{10} - 66z^5 + 1}$ 
					& 0.40
					& 0.06
					& 0.65
					& $z, -1/z$
					& $\mathfrak{C}_2$\\
				\hline 
				$345025251z^6$ 
					& 300.63
					& 0.02
					& 0.07
					& $z, 1/({2601z})$
					&$\mathfrak{C}_2$\\
				\hline
			\end{tabular}
			\caption{Running times for {\texttt{automorphism\_group\_QQ}} on rational functions with nontrivial automorphism group. \label{table:nontrivial_auts}}
		\end{table}}		
		\vspace{.2in}
			
		{
		\begin{table}[!ht]
			\renewcommand{\arraystretch}{1.5}
			\begin{tabular}{|c|c|c|c|c|c|}
				\hline
				$k$ & $f$ & CRT & FP & GB & $\Aut_{\phi^f}(\Q)$  \\
				\hline
				3 & $\frac{3z - 7}{5z - 1}$ 
					& 6.19 & 0.05 & 0.01 &
					$z, \frac{z + 5}{3z - 1},
					\frac{-19z + 21}{-5z + 19}, 
					\frac{11z - 29}{13z - 11}$\\
					\hline
				-3 & $\frac{-3z}{-3z - 4}$ 
					& 9.94 & 0.03 & 0.03 &
					$z, \frac{z}{2z - 1},\frac{-9z + 9}{7z + 9}, 
					\frac{-9z + 9}{-25z + 9}$\\ \hline
				6 & $\frac{7z + 10}{-3z + 8}$ 
					& 3. 88 & 0.04 & 0.29 &
					$z,\frac{101z - 51}{55z - 101}$\\ \hline
				-6 & $\frac{-7z - 7}{-3z + 1}$
					& 0.12 & 0.03 & 0.42 
					& $z, \frac{7z}{4z - 7}$\\ \hline
				9 & $\frac{z - 8}{4z - 10}$ & 
					90.27 & 0.17 & 3.31 &	
					$z,\frac{84z - 65}{116z - 84}, 
					\frac{-21z + 8}{-40z + 21},
					\frac{-76z + 63}{-84z + 76}$ \\ \hline
				-9 & $\frac{8z + 1}{-2z + 9}$ 
					& 101.20 & 0.17 & 19.25 &
					$z,\frac{25z + 63}{77z - 25}, 
					\frac{-35z + 8}{18z + 35},
					\frac{7z + 65}{-85z - 7}$ \\ \hline
				12 & $\frac{-2z - 10}{4z + 1}$
					& 0.80 & 0.28 & 29.96 & 
					$z,\frac{-2z - 96}{-15z + 2}$\\ \hline
				-12 & $\frac{-3z}{-5z + 1}$ 
					& 0.29 & 0.26 & 22.98
					& $z,\frac{5z - 3}{8z - 5}$\\ \hline
				15 & $\frac{z + 9}{-z + 1}$ 
					& 16.66 & 0.65 & 100.93 &
					$z, -z + 8,\frac{4z + 9}{z - 4},
					\frac{4z - 41}{z - 4}$\\ \hline
				-15 & $\frac{-4z - 1}{8z - 8}$ 
					& 32.68 & 2.39 & 371.42 &
					$z, -z - \frac38,\frac{-3z + 1}{16z + 3}, 
					\frac{-24z - 17}{128z + 24}$\\ \hline
				18 & $\frac{z + 10}{5z + 10}$ 
					& 4.42 & 2.87 & 739.89 &
					$z,\frac{95z - 99}{75z - 95}$\\ \hline
				-18 & $\frac13(2z - 5)$ 
					& 0.97 & 1.42 & 4.23 &
					$z,\frac{-5z - 7}{3z + 5}$\\ \hline
			\end{tabular}
			\caption{Running times for {\texttt{automorphism\_group\_QQ}} on $\phi^f$ where $\phi(z) = z^k$. Automorphism groups are either $\Z/2$ or $\Z/2\times\Z/2$.  
\label{table:ConjugatesOfPowerMaps}}
		\end{table}}
		\begin{table}[!hb]
			{\small
			\begin{tabular}{|c|c|c|c|c|c|c|c|c|}
				\hline
				 \multirow{2}{*}{d} & & \multicolumn{6}{|c|}{Height Bound}\\
				& &  \multicolumn{1}{c}{$50$} &\multicolumn{1}{c} {$10^2$} 
				& \multicolumn{1}{c}{$10^3$} & \multicolumn{1}{c}{$10^4$} 
				& \multicolumn{1}{c}{$10^5$} & \multicolumn{1}{c|}{$10^6$}\\
				\hline
				 3 & CRT 
					& 0.057
					& 0.094
					& 0.101
					& 0.102
					& 0.093
					& 0.103 \\
				   &  FP
					& 0.011
					& 0.012
					& 0.011
					& 0.010
					& 0.010
					& 0.010 \\
				   &  GB
					& 0.010
					& 0.010
					& 0.020
					& 0.020
					& 0.030
					& 0.030\\
					\hline
				 6 & CRT
					& 0.083
					& 0.075
					& 0.067
					& 0.121
					& 0.109
					& 0.109 \\
				   &  FP
					& 0.021
					& 0.022
					& 0.023
					& 0.024
					& 0.024
					& 0.025 \\
				   &  GB
					& 0.660
					& 0.895
					& 1.325
					& 1.910
					& 2.495
					& 3.320 \\
					\hline
				 9 & CRT
					& 0.108
					& 0.177
					& 0.148
					& 0.164
					& 0.138
					& 0.156 \\
				   &  FP
					& 0.069
					& 0.072
					& 0.073
					& 0.080
					& 0.082
					& 0.088 \\
				   &  GB
					& 4.550
					& 5.015
					& 7.235
					& 10.115
					& 11.735
					& 13.800 \\
					\hline
				12 & CRT
					& 0.162
					& 0.258
					& 0.278
					& 0.263
					& 0.264
					& 0.301 \\
				   & FP
					& 0.234
					& 0.230
					& 0.247
					& 0.260
					& 0.280
					& 0.307 \\
					\hline
				15 & CRT
					& 0.456
					& 0.500
					& 0.479
					& 0.496
					& 0.454
					& 0.436 \\
				   & FP
					& 0.788
					& 0.789
					& 0.776
					& 0.772
					& 0.793
					& 0.824 \\
					\hline
				18 & CRT
					& 0.973
					& 0.985
					& 0.954
					& 0.997
					& 1.026
					& 1.105 \\
				   & FP
					& 1.913
					& 1.884
					& 2.008
					& 2.160
					& 2.188
					& 2.322 \\
					\hline
				21 & CRT 
					& 2.419
					& 2.437
					& 2.634
					& 2.285
					& 2.433
					& 2.357\\
				   & FP
					& 5.112
					& 5.106
					& 5.377
					& 5.399
					& 5.226
					& 5.078\\
				\hline
				\end{tabular}}
			\caption{Median running times for the three algorithms on $100$ random rational functions with given degree and height bound.\label{table:fixed_points}}
		\end{table}


\bibliographystyle{amsalpha}
\bibliography{xander_bib}

\end{document}